\documentclass[11pt]{amsart}
\usepackage{mathrsfs}
\usepackage{amsfonts}
\usepackage{amssymb}
\usepackage{amsxtra}
\usepackage{dsfont}
\usepackage{color,xcolor}
\usepackage{graphicx}
\usepackage{enumitem}
\usepackage[english,french, polish]{babel}

\usepackage{pgf,tikz}
\usetikzlibrary{arrows}

\usepackage[T1]{fontenc}
\usepackage{newtxmath}
\usepackage{newtxtext}

\usepackage[compress, sort]{cite}

\usepackage[margin=2.5cm, centering]{geometry}
\usepackage[colorlinks,citecolor=blue,urlcolor=blue,bookmarks=true]{hyperref}
\hypersetup{
pdfpagemode=UseNone,
pdfstartview=FitH,
pdfdisplaydoctitle=true,
pdfborder={0 0 0}, 
pdftitle={Construction of Brooks--Lindenstrauss kernels on affine buildings of arbitrary reduced type, with applications},
pdfauthor={Jean-Philippe Anker and Bertrand R\'emy and and Bartosz Trojan},
pdflang=en-US
}

\newcommand{\pl}[1]{\foreignlanguage{polish}{#1}}
\newcommand{\fr}[1]{\foreignlanguage{french}{#1}}

\newcommand{\CC}{\mathbb{C}}
\newcommand{\DD}{\mathbb{D}}

\newcommand{\NN}{\mathbb{N}}

\newcommand{\QQ}{\mathbb{Q}}
\newcommand{\RR}{\mathbb{R}}

\newcommand{\ZZ}{\mathbb{Z}}

\newcommand{\calC}{\mathcal{C}}
\newcommand{\calL}{\mathcal{L}}

\newcommand{\calF}{\mathcal{F}}

\newcommand{\calO}{\mathcal{O}}

\newcommand{\calH}{\mathcal{H}}
\newcommand{\calP}{\mathcal{P}}
\newcommand{\scrX}{\mathscr{X}}
\newcommand{\scrC}{\mathscr{C}}
\newcommand{\scrA}{\mathscr{A}}

\newcommand{\mathfrakA}{\mathfrak{a}}

\newcommand{\bfc}{\mathbf{c}}

\newcommand{\St}{\operatorname{St}}

\newcommand{\Aut}{\operatorname{Aut}}
\newcommand{\GL}{\operatorname{GL}}
\newcommand{\Aff}{\operatorname{Aff}}
\newcommand{\height}{\operatorname{ht}}
\newcommand{\Id}{\operatorname{Id}}
\newcommand{\frakp}{\mathfrak{p}}
\newcommand{\vphi}{\varphi}
\renewcommand{\Re}{\operatorname{Re}}

\newcommand{\norm}[1]{\lvert {#1} \rvert}
\newcommand{\abs}[1]{\lvert {#1} \rvert}
\newcommand{\sprod}[2]{\langle {#1}, {#2} \rangle}

\newcommand{\Tr}{\operatorname{Tr}}

\newcommand{\co}{\operatorname{co}}

\newcommand{\cspan}{\operatorname{\mathbb{C}-span}}
\newcommand{\zspan}{\operatorname{\mathbb{Z}-span}}

\newcommand{\ind}[1]{{\mathds{1}_{{#1}}}}

\renewcommand{\atop}[2]{\stackrel{{#1}}{{#2}}}

\allowdisplaybreaks

\theoremstyle{plain}
\newcounter{thm}

\newtheorem{main_theorem}[thm]{Theorem}

\newtheorem{theorem}{Theorem}[section]

\newtheorem{lemma}[theorem]{Lemma}

\newtheorem{claim}[theorem]{Claim}
\newtheorem{remark}[theorem]{Remark}
\newtheorem*{theorem*}{Theorem}

\newcommand{\thistheoremname}{}
\newtheorem*{genericthm*}{\thistheoremname}
\newenvironment{namedthm*}[1]
{\renewcommand{\thistheoremname}{#1}%
\begin{genericthm*}}
{\end{genericthm*}}

\title[Brooks--Lindenstrauss kernels on arbitrary reduced buildings]
{Construction of Brooks--Lindenstrauss kernels on affine buildings of arbitrary reduced type, with applications}

\author{Jean-Philippe Anker}
\address{
	\fr{
	Jean-Philippe Anker\\
	Universit\'e d’Orl\'eans, Universit\'e de Tours, CNRS,
    IDP, UMR 7013,
    Orl\'eans, France
	}
}
\email{anker@univ-orleans.fr}

\author{Bertrand R\'emy}
\address{
	\fr{
	Bertrand R\'emy\\
	Unit\'e de Math\'ematiques Pures et Appliquees (UMR 5669),
	\'Ecole normale sup\'erieure de Lyon,
	site Monod,
	46 all\'ee d'Italie,
	69364 Lyon Cedex 07,
	France
	}
}
\email{bertrand.remy@ens-lyon.fr}

\author{Bartosz Trojan}
\address{
	\pl{
	Bartosz Trojan\\
    Wydzia\l{} Matematyki\\
	Politechnika Wroc\l{}awska\\
	Wyb. Wyspia\'{n}skiego 27\\
	50-370 Wroc\l{}aw\\
	Poland}
}
\email{bartosz.trojan@gmail.com}

\begin{document}
\selectlanguage{english}

\begin{abstract}
	This article deals with harmonic analysis on affine buildings. Its main goal is to construct suitable kernels associated to
	a discrete multitemporal wave equation on the latter spaces, the long-standing motivation being to contribute to progress
	in arithmetic quantum unique ergodicity (AQUE) on certain Riemannian manifolds. 
\end{abstract}

\maketitle

\section*{Introduction}

Quantum unique ergodicity (QUE) deals with certain limit measures on compact (or finite-volume) negatively curved manifolds 
or locally symmetric manifolds of non-compact type. More precisely, on a negatively curved compact manifold $M$ we consider 
a sequence $(\varphi_j : j\in\NN)$ of $L^2$-normalized eigenfunctions of the positive Laplacian $\Delta$ corresponding to 
eigenvalues $\lambda_j\to\infty$. In this context, Z.~Rudnick and P.~Sarnak \cite{RudnickSarnak} conjectured that the only
weak-$\star$ cluster value of the sequence of probability measures $|(\varphi_j)_{j}|^2{\rm dvol}_M$ is ${\rm dvol}_M$ itself. 
When $M$ is a locally symmetric space whose universal cover $\widetilde M=G/K$ is an irreducible Riemannian symmetric space of
non-compact type, the analogous conjecture has a more group-theoretic formulation since then the identity component
$G={\rm Isom}(\widetilde M)^\circ$ is a simple Lie group and $K$ is a maximal compact subgroup (they are all conjugated).
The conjectured statement on limit measures is the same but one makes the assumption that the functions 
$(\varphi_j : j \in \NN)$ are joint eigenfunctions of the algebra $\DD$ of all $G$-invariant differential operators on
$\widetilde M$ descended to $M=\Gamma\backslash\widetilde{M}$, where $\Gamma$ is the fundamental group of $M$.

Though it is possible to attack these conjectures from a purely analytic perspective (see \emph{e.g.} \cite{Anantharaman22}), in 
the locally symmetric case some specific choices of eigenfunctions $(\varphi_j : j \in \NN)$ can be made, relying on group- 
and representation-theoretic considerations. Indeed, if one assumes that $\Gamma$, the fundamental group of $M$, is an arithmetic
lattice in $G$, then an additional commutative algebra of difference operators becomes available. More precisely, since $G$ can
be seen as the Archimedean points of an algebraic group ${\bf G}$ over a number field $F$, there is a Hecke algebra
$\mathscr{H}$ associated with an adelic realization of the homogeneous space $\Gamma\backslash G$ whose elements commute with
$\DD$. Thus it makes sense to work with joint eigenfunctions $\varphi_j$ of both $\DD$ and $\mathscr{H}$, which is the context
of arithmetic quantum unique ergodicity (AQUE). These additional arithmetic invariance conditions may be considered as a strong 
restriction, but they are automatically fulfilled if the $\DD$-spectrum is multiplicity free, which is actually conjectured under
suitable assumptions. The most striking and complete results in AQUE are those obtained by E.~Lindenstrauss for compact
hyperbolic surfaces \cite{Lindenstrauss} together with the finite volume case proved by K.~Soundararajan \cite{Soundararajan},
but more and more cases became available in recent years. 

Still in the arithmetic situation, there is an intermediate context in which the additional arithmetic invariance assumption
only involves a single place $v$ of $F$. In this case, the corresponding subalgebra $\mathscr{H}_v$ of $\mathscr{H}$ is a
convolution algebra of functions on the non-Archimedean Lie group ${\bf G}(F_v)$ where $F_v$ denotes the completion of $F$
with respect to $v$. Equivalently, $\mathscr{H}_v$ can be seen as some commutative algebra $\mathscr{A}_0$ of averaging
operators acting on functions defined on the vertices of the Bruhat--Tits building attached to ${\bf G}(F_v)$. This roughly
explains why the main technique of the present paper is the purely geometric (building-theoretic) counterpart to harmonic
analysis of bi-$K_v$-invariant functions on ${\bf G}(F_v)$ where $K_v$ is a suitably chosen maximal compact subgroup of
${\bf G}(F_v)$, as developed by I.~Satake and I.G.~Macdonald \cite{macdo0}.

We need to introduce additional concepts in order to state our results whose main outcome is the construction of a wave
kernel on an affine building. Let us recall that an affine building \cite{TitsCorvallis79} is a simplicial complex covered
by subcomplexes called \emph{apartments} which are all isomorphic to a given Euclidean tiling. The (discrete) symmetry group
of the model tiling is called the \emph{Weyl group} of the building, and apartments in the building are requested to satisfy
the following incidence conditions: 
\begin{itemize} 
	\item[--] any two simplices must be contained in an apartment;
	\item[--] given any two apartments, there must be an isomorphism between them fixing their intersection. 
\end{itemize}
These axioms allow one to define a real-valued distance between arbitrary points making an affine building a non-positively
curved contractible space. They also lead to a vectorial distance between vertices (\emph{i.e.} $0$-simplices) of particular type
and called special vertices: A vertex is called \emph{special} if its stabilizer in the Weyl group (of any apartment containing
it) is the full linear part of the Weyl group; we denote by $V_s$ the set of special vertices. 

Given an affine building $\scrX$ a (purely geometric) harmonic analysis can be developed and leads to a Gelfand--Fourier 
transform with no group involved (see Subsection \ref{ss - analysis on buildings} and \cite{park2} for more details).
First of all, the geometry of an apartment embedded in $\scrX$ defines a root system $\Phi$ with Weyl group $W$.
Let us denote by $\mathfrak a$ the Euclidean space containing it; this space contains the co-weight lattice $P$ in which
the vectorial distance $\sigma : V_s \times V_s \to P^+$ takes its values, see Section \ref{ss - affine buildings} for details.
For the purpose of applications, we consider a commutative Banach algebra $\mathscr{A}_1$ densely containing $\mathscr{A}_0$
whose spectrum $\Sigma$ is contained in the complexification $\mathfrakA_\CC$ of $\mathfrakA$, so that the Gelfand--Fourier 
transform $\mathscr{F}$ maps $\mathscr{A}_1$ into the algebra $\mathcal{C}(\Sigma)$ of continuous functions on $\Sigma$,
which is defined as \eqref{eq:1:4}.
\begin{main_theorem}
	\label{th - kernel intro}
	Assume that the root system $\Phi$ of the affine building $\scrX$ is reduced.
	Set $\rho=\sum_{i=1}^r\lambda_i$ where $\big\{\lambda_1, \ldots, \lambda_r\big\}$ is the canonical co-weight basis in
	$\mathfrak a$. Let $z_0 = \zeta_0 + i\theta_0$ be an element of $\Sigma$. Then there exist $M(\zeta_0), N(\zeta_0)
	\in \NN$ such that for all integers $M \geqslant M(\zeta_0)$ and $N \geqslant N(\zeta_0)$, there is a function
	$k: V_s \times V_s \rightarrow \CC$ with the following properties:
	\begin{enumerate}[label=\rm (\roman*), start=1, ref=\roman*]
		\item
		\label{en:0:1}
		for all $x, y \in V_s$ the value of $k(x, y)$ depends only on the vectorial distance $\sigma(x, y)$; moreover, 
		$k(x, y)$ vanishes when
		\[
			\sigma(x, y)\not\preceq cM^{r+1}N^r\rho
		\]
		where $c>0$ is a constant depending only on the root system $\Phi$;
		\item
		\label{en:0:2}
		there are $C > 0$ and $c' > 0$ such that
		\[
			\sup_{x, y \in V_s} |k(x, y)| \leqslant C e^{-c' N};
		\]
		\item
		\label{en:0:3}
		for all $x \in V_s$ and $z \in \Sigma$, we have 
		\[
			\calF\!\big(k(x, \cdot )\big)(z) \geqslant -1 \quad \text{while} \quad \calF\!\big(k(x, \cdot )\big) (z_0) 
			\geqslant M.
		\]
	\end{enumerate}
\end{main_theorem}
Such kernels were first constructed by Sh.~Brooks and E.~Lindenstrauss in \cite{BrooksLindenstrauss2010, BrooksLindenstrauss2013,
BrooksLindenstrauss2014} for homogeneous trees and next by Z.~Shem-Tov in \cite{Shem-Tov2022} for affine buildings of type
$\tilde{A}_r$ with $r \geqslant 2$. Our construction works for any affine building of reduced type and our kernel differs from 
the one in \cite{Shem-Tov2022}.

Let us briefly comment on the proof of Theorem \ref{th - kernel intro}. To construct the kernel we seek for building blocks
satisfying \eqref{en:0:1} and \eqref{en:0:2}. For this purpose in the case of homogeneous trees, Sh.~Brooks and E.~Lindenstrauss
used a fundamental solution to a discrete wave equation. Notice that, on real hyperbolic spaces, fundamental solutions 
to the shifted wave equation are known to possess similar properties (see for instance \cite{Tataru2001}). Therefore we introduce
a discrete multitemporal wave equation on affine buildings which is the analog of the system studied by 
M.A.~Semenov-Tian-Shansky \cite{Semenov1976} on non-compact symmetric spaces. A systematic study of the system, including
 scattering theory, is the subject of the forthcoming paper \cite{AnkerAtAll2026}.
In this article, we consider a particular non-trivial solution 
to this equation and show that it satisfies finite propagation speed and uniform exponential space-time decay 
(see Theorem \ref{thm:2:3}). Inspired by Sh.~Brooks and E.~Lindenstrauss 
\cite{BrooksLindenstrauss2010, BrooksLindenstrauss2013, BrooksLindenstrauss2014}, we define a kernel satisfying 
Theorem \ref{th - kernel intro}\eqref{en:0:1}--\eqref{en:0:3}. The last property is proved by a careful study of its Gelfand--Fourier transform on
the spectrum $\Sigma$. Lastly let us note again that the statement of Theorem \ref{th - kernel intro} and its proof make
no use of any group action.

Let us explain now how the above result fits in a general strategy elaborated to prove some cases of AQUE. 
The general scheme of proof introduced by E.~Lindenstrauss \cite{Lindenstrauss} consists of three steps which are quite 
independent from one another (see \cite[Introduction]{SV19} for more details). 
\begin{itemize}
	\item[$\bullet$] {\bf Microlocal lift.}
	This step consists in lifting the problem from the locally symmetric space $M=\Gamma\backslash G/K$ to  a geometric bundle 
	over $M$ (the unit tangent bundle in rank one) and eventually to the homogeneous space $\Gamma \backslash G$. 
	However, even if initially it was proved geometrically, there is now a representation-theoretic approach which works in a 
	fairly general setup \cite{SV07}. The outcome is a reformulation of the problem where the main objects of study are limits
	$\mu=\lim_{j\to\infty}|\varphi_j|^2 {\rm dvol_{\Gamma \backslash G}}$, where ${\rm dvol}_{\Gamma \backslash G}$ is the 
	normalized Haar measure on $\Gamma\backslash G$,
	and where $\varphi_j$ are $L^2$-normalized eigenfunctions of both $\mathscr{Z}(\mathfrak{g})$ and $\mathscr{H}$
	(or part of it).  Here $\mathfrak{g}$ denotes the Lie algebra of $G$, and $\mathscr{Z}(\mathfrak{g})$ the center of its
	universal enveloping algebra.
\item[$\bullet$] {\bf Positive entropy.} 
	A consequence of the previous step is that any cluster value $\mu$ of $( |\varphi_j|^2 {\rm d}x : j \in \NN)$ is 
	\emph{invariant} under a maximal split torus of $G$, denoted by $A$ in the sequel. In this step, the goal is to prove that 
	for all regular elements $a \in A$ (or at least for some specific ones), almost every ergodic component of a limit measure 
	$\mu$ has positive entropy. This is the place where our construction may play a role. 
\item[$\bullet$] {\bf Classification of homogeneous measures.}
	The next step consists in exploiting results from a wide program which is interesting in itself, namely the classification of
	invariant and ergodic probability measures on homogeneous spaces such as $\Gamma\backslash G$. This program is far from being
	completed and it is usually the main limiting step in proving instances of AQUE. Some intermediate results often state that
	$A$-invariant measures satisfying suitable positivity entropy conditions are measures associated to homogeneous spaces of 
	algebraic subgroups of $G$: the latter measures are usually called \emph{algebraic}. 
\end{itemize}
In this general scheme, the use of a kernel as in Theorem \ref{th - kernel intro} occurs in the second step, following ideas 
initially due to Sh.~Brooks and E.~Lindenstrauss \cite{BrooksLindenstrauss2010,BrooksLindenstrauss2013,BrooksLindenstrauss2014}
and enables one to deal with suitable "explicit" partitions of $\Gamma\backslash G$. This leads to the following positive entropy
result whose proof consists of checking that the lemmata used by Z.~Shem-Tov in \cite[Theorem 1.1]{Shem-Tov2022} to deal with
the case $G = {\rm SL}_n(\RR)$ can be used with minor modifications once his specific kernel for ${\rm SL}_n$ is replaced by ours. 

Let us consider the following setup for next theorem (see Subsection \ref{ss - potential applications} for more details): 
Let ${\bf G}$ be an absolutely almost simple group over $\QQ$ and let $p$ be a prime number. We set $G_\infty = {\bf G}(\RR)$
and $G_p = {\bf G}(\QQ_p)$, we choose a (special) maximal compact subgroup $K_p$ in $G_p$. Let us denote by $\Gamma[\smash{\frac1p}]$ 
(resp.~by $\Gamma$) the $\ZZ[\smash{\frac1p}]$-points (resp.~the $\ZZ$-points) of ${\bf G}$, for instance defined by means of
some embedding ${\bf G}<{\rm GL}_m$. In good cases (\emph{e.g.} when ${\bf G}$ is simply connected), we have
$G_p=\Gamma[\smash{\frac1p}] K_p$ and $\Gamma[\smash{\frac1p}]\cap K_p=\Gamma$, hence a bijection 
$\Gamma g_\infty\mapsto\Gamma[\smash{\frac1p}](g_\infty, 1)K_p$ between $\Gamma\backslash G_\infty$ and 
$\Gamma[\frac1p]\backslash(G_\infty\times G_p)/K_p$. As recalled in Subsection \ref{ss - general setting}, since we are 
eventually interested in $G_\infty$ only, we can adjust the choice of ${\bf G}$ so that $G_p$ is split over $\QQ_p$, hence 
the Bruhat--Tits building of $G_p$ has a reduced root system and thus admits kernels as in Theorem \ref{th - kernel intro}.
This will lead to a non-compact but finite volume quotient space $\Gamma \backslash G_\infty$. To deal with compact spaces 
$\Gamma\backslash G_\infty$ one has to assume that ${\bf G}$ is anisotropic over $\QQ$. This can be achieved either by choosing 
${\bf G}$ so that $G_{p'}$ is compact at some other non-Archimedean place $p'$ (implying strong restrictions on the absolute 
type of ${\bf G}$) or by assuming that the rational points of ${\bf G}$ at some Archimedean place are compact (with no
restriction on the absolute root system but requiring to pass from $\QQ$ to a number field $F$ for the ground field of
${\bf G}$). For simplicity we write $G$ in place of $G_\infty$.

\begin{main_theorem}
	\label{th - positive entropy intro}
	Let $G$ and $\Gamma$ be as in the paragraph above. Assume that $G$ is $\RR$-split. 
	Let $\mu$ be an $A$-invariant probability measure on the homogeneous space $X = \Gamma \backslash G$. Assume that
	$\mu=c \lim_{j\to\infty}|\varphi_j|^2 \: {\rm dvol_{\Gamma \backslash G}}$ is a weak-$\star$ limit where $c>0$ and 
	$\varphi_j$ are some $L^2$-normalized joint eigenfunctions of both $\mathscr{Z}(\mathfrak{g})$ and of the Hecke algebra 
	at some fixed prime $p$. Then any regular element $a\in A$ acts with positive entropy on each ergodic component of $\mu$.
\end{main_theorem}

Once this positive entropy result is proved, it can then be combined with classifications of $A$-invariant measures as alluded 
to above. 
Note that $A$-invariance of the limit measure is used here because the positive entropy criterion from Sect. \ref{subsection - pf ThB} involves a measure invariant under the action of a regular element of a maximal split torus. 
The $A$-invariance itself is implied for instance by a mild non-degenaracy assumption made on the sequence of eigenfunctions (see \cite[Theorem 1.6(3)]{SV07}).
Therefore, Theorem \ref{th - positive entropy intro} may contribute to provide some additional cases of AQUE involving 
a single place.

The structure of the paper goes as follows: In Section \ref{s - affine buildings}, we introduce the basics about affine buildings
and harmonic analysis thereon (again, we do it without calling on Bruhat--Tits theory, in fact without using any group action).
In Section \ref{s - multitemporal wave equation}, we define a discrete multitemporal wave equation and use Fourier analysis to 
study a particular solution for which we prove finite propagation speed and uniform space-time estimates. In Section 
\ref{s - kernel construction}, we give a detailed proof of Theorem \ref{th - kernel intro}. In Section \ref{s - applications}, 
we first describe the general context in which the constructed kernels could be used and then prove Theorem \ref{th - positive entropy intro}.

We conclude the introduction with some notation. In the whole paper, $\NN$ (resp.~$\NN^*$) denotes the set of integers $\geqslant 0$
(resp.~$>0$).

\section*{Acknowledgement}
J.-Ph. Anker and B. Trojan thank \'Ecole normale sup\'erieure de Lyon for invitations; B. R\'emy and B. Trojan thank Universit\'e
d'Orl\'eans for the same reason. B. Trojan acknowledges financial support from CNRS for a research trimester in 2023 in 
Orl\'eans.

\section{Affine buildings}
\label{s - affine buildings} 

\subsection{Buildings}
\label{ss - buildings} 
A family $\scrX$ of non-empty finite subsets of some set $V$ is an \emph{abstract simplicial complex}
if for all $\sigma \in \scrX$, each subset $\gamma \subseteq \sigma$ also belongs to $\scrX$. The elements of $\scrX$
are called \emph{simplices}. The dimension of a simplex $\sigma$ is $|\sigma| - 1$. Zero dimensional simplices are called
\emph{vertices}. The set $V(\scrX) = \bigcup_{\sigma \in \scrX} \sigma$ is the \emph{vertex set} of $\scrX$.
The dimension of the complex $\scrX$ is the maximal dimension of its simplices. A \emph{face} of a simplex $\sigma$
is a non-empty subset $\gamma \subseteq \sigma$. For a simplex $\sigma$  we denote by $\St(\sigma)$ the collection of
simplices containing $\sigma$; in particular, $\St(\sigma)$ is a simplicial complex. Two abstract simplicial complexes
$\scrX$ and $\scrX'$ are \emph{isomorphic} if there is a bijection $\psi: V(\scrX) \rightarrow V(\scrX')$ such that for
all $\sigma = \{x_1, \ldots, x_k\} \in \scrX$ we have $\psi(\sigma) = \{\psi(x_1), \ldots, \psi(x_k)\} \in \scrX'$.
With every abstract simplicial complex $\scrX$ one can associate its \emph{geometric realization} $|\scrX|$ in the vector 
space of functions $V \rightarrow \RR$ with finite support, see \emph{e.g.} \cite[\S2]{Munkres1996}.

A set $\scrC$ equipped with a collection of equivalence relations $\{\sim_i : i \in I\}$ where $I = \{0, \ldots, r\}$,
is called a \emph{chamber system} and the elements of $\scrC$ are called \emph{chambers}. A \emph{gallery} of type
$f = i_1 \ldots i_k$ in $\scrC$ is a sequence of chambers $(c_1, \ldots, c_k)$ such that for all
$j \in \{1,2, \ldots k\}$, we have $c_{j-1} \sim_{i_j} c_j$, and $c_{j-1} \neq c_j$. If $J \subseteq I$, a \emph{residue} of 
type $J$ is a subset of $\scrC$ such that any two chambers can be joined by a gallery of type $f = i_1 \ldots i_k$ with
$i_1, \ldots, i_k \in J$. From a chamber system $\scrC$ we can construct an abstract simplicial complex where each
residue of type $J$ corresponds to a simplex of dimension $r - |J|$. Then, for a given vertex $x$, we denote by
$\calC(x)$ the set of chambers containing $x$. 

A \emph{Coxeter group} is a group $W$ given by a presentation
\[
	\left\langle
	r_i : (r_i r_j)^{m_{i, j}} = 1, \text{ for all } i, j \in I
	\right\rangle
\]
where $M = (m_{i,j})_{I \times I}$ is a symmetric matrix with entries in $\ZZ \cup \{\infty\}$ such that for all
$i, j \in I$,
\[
	m_{i,j}=
	\begin{cases}
		\geqslant 2 & \text{if }i\neq j,\\
		1 & \text{if }i=j.
	\end{cases}
\]
For a word $f=i_1 \cdots i_k$ in the free monoid $I$ we denote by $r_f$ an element of $W$ of the form
$r_f= r_{i_1} \cdots r_{i_k}$. The length of $w \in W$, denoted $\ell(w)$, is the smallest integer $k$ such that there
is a word $f=i_1 \cdots i_k$ and $w=r_f$. We say that $f$ is reduced if $\ell(r_f) = k$. A Coxeter group $W$ may be turned
into a chamber system by introducing in $W$ the following collection of equivalence relations:
$w \sim_i w'$ if and only if $w = w'$ or $w = w' r_i$. The corresponding simplicial complex $\Sigma$ is called 
\emph{Coxeter complex}. 

A simplicial complex $\scrX$ is called a \emph{building of type $\Sigma$} if it contains a family of subcomplexes called
\emph{apartments} such that
\begin{enumerate}[start=1, label=\rm (B\roman*), ref=B\roman*]
	\item\label{en:3:1}
	each apartment is isomorphic to $\Sigma$,
	\item\label{en:3:2}
	any two simplices of $\scrX$ lie in a common apartment,
	\item\label{en:3:3}
	for any two apartments $\scrA$ and $\scrA'$ having a chamber in common there is an 
	isomorphism $\psi: \scrA \rightarrow \scrA'$ fixing $\scrA \cap \scrA'$ pointwise.
\end{enumerate}
The rank of the building is the cardinality of the set $I$. We always assume that $\scrX$ is irreducible. A simplex $c$ is a 
chamber in $\scrX$ if it is a chamber in any of its apartments. By $C(\scrX)$ we denote the set of chambers in $\scrX$. Using
the building axioms we see that $C(\scrX)$ has a chamber system structure. However, it is not unique. A geometric realization
of the building $\scrX$ is its geometric realization as an abstract simplicial complex. In this article we assume that the 
system of apartments in $\scrX$ is \emph{complete} meaning that any subcomplex of $\scrX$ isomorphic to $\Sigma$ is an apartment.
We denote by $\Aut(\scrX)$ the group of automorphisms of the building $\scrX$.

\subsection{Affine Coxeter complexes}
\label{sec:1}
In this section we recall basic facts about root systems and Coxeter groups. A general reference is \cite{Bourbaki2002},
which deals with Coxeter systems attached to reduced root systems. Since at the beginning the root system may be non-reduced
we will also refer to \cite{mz1, park}.

Let $\Phi$ be an irreducible, but not necessarily reduced, finite root system in a Euclidean space $\mathfrakA$.
The $\zspan$ of $\Phi$ is the \emph{root lattice} $L$. Let $\{\alpha_1,\dots,\alpha_r\}$ be a base of $\Phi$, and let $\Phi^+$
denote the corresponding subset of \emph{positive roots}. We denote 
\[
	\mathfrakA_+=\bigl\{x\in\mathfrakA:\langle\alpha,x\rangle>0\text{ for every }\alpha\in\Phi^+\bigr\}
\]
the \emph{positive Weyl chamber}. 

For a positive root $\alpha \in \Phi^+$, its \emph{height} is defined as
\[
	\height(\alpha)=\sum_{i=1}^rn_i 
\]
wherever $\alpha=\sum_{i=1}^rn_i\alpha_i$ with $n_i\in\NN$. Since $\Phi$ is irreducible, there is a unique \emph{highest root}
$\alpha_0=\sum_{i=1}^rm_i\alpha_i$ with $m_i\in\NN$. We set
\[
	I_0=\{1,\dots,r\}
	\quad\text{and}\quad
	I_g=\{0\}\cup\{i\in I_0:m_i=1\}.
\]
Let $\Phi\spcheck=\{\alpha\spcheck =\frac2{\sprod{\alpha}{\alpha}} \alpha :\alpha\in\Phi\}$ be the dual root system.
The $\zspan$ of $\Phi\spcheck$ is the \emph{co-root lattice} $Q$. Let $Q^+=\sum_{\alpha\in\Phi^+}\NN\alpha\spcheck$.
The dual basis $\{\lambda_1,\dots,\lambda_r\}$ to $\{\alpha_1,\dots,\alpha_r\}$ consists of \emph{fundamental co-weights} and
its $\zspan$ is the \emph{co-weight lattice} $P$. Let $P^+=\sum_{i=1}^r\NN\lambda_i$ be the cone of \emph{dominant co-weights},
and let $\sum_{i=1}^r\NN^*\lambda_i$ be the subcone of \emph{strongly dominant co-weights}. For instance
\[
	\rho=\sum_{i=1}^r\lambda_i
\]
is a strongly dominant co-weight. Notice that
\[
	\rho=\frac12\sum_{\alpha\in{\Phi^{++}}}\alpha\spcheck
\]
where $\Phi^{++}$ denotes the set of indivisible positive roots in $\Phi$, see \cite[Proposition 10.29, \S 1, VI]{Bourbaki2002}

Let $\calH$ be the family of affine hyperplanes, called \emph{walls}, being of the form
\[
	H_{\alpha; k} = \big\{x \in \mathfrakA : \langle x, \alpha \rangle = k \big\}
\]
where $\alpha \in \Phi^+$ and $k \in \ZZ$. Each wall determines two half-apartments
\[
	H^-_{\alpha; k} = \big\{x \in \mathfrakA : \langle x, \alpha \rangle \leqslant k\big\}
	\quad\text{and}\quad
	H^+_{\alpha; k} = \big\{x \in \mathfrakA : \langle x, \alpha \rangle \geqslant k\big\}.
\]
Note that for a given $\alpha$, the family $H^-_{\alpha; k}$ is increasing in $k$ while the family $H^+_{\alpha; k}$ is
decreasing. To each wall we associate $r_{\alpha; k}$ the orthogonal reflection in $\mathfrakA$ defined by 
\[
	r_{\alpha; k}(x) = x - \big(\sprod{x}{\alpha} - k\big)\alpha\spcheck.
\]
Set $r_0 = r_{\alpha_0; 1}$, and $r_i = r_{\alpha_i; 0}$ for each $i \in I_0$. 

The \emph{finite Weyl group} $W$ is the subgroup of $\GL(\mathfrakA)$ generated by $\{r_i: i \in I_0\}$. Let us denote
by $w_0$ the longest element in $W$. The \emph{fundamental sector} in $\mathfrakA$ defined as
\begin{align*}
	S_0 = \overline{\mathfrakA_+}
	&= \bigoplus_{i \in I_0} \RR_+ \lambda_i \\
	&= \bigcap_{i \in I_0} H^+_{\alpha_i; 0}
\end{align*}
is the fundamental domain for the action of $W$ on $\mathfrakA$.

The \emph{affine Weyl group} $W^a$ is the subgroup of $\Aff(\mathfrakA)$ generated by $\{r_i: i \in I\}$. Observe that 
$W^a$ is a Coxeter group. The hyperplanes $\calH$ give the geometric realization of its Coxeter complex $\Sigma_\Phi$. To see
this, let $C(\Sigma_\Phi)$ be the family of closures of the connected components of
$\mathfrakA \setminus \bigcup_{H \in \calH} H$. By $C_0$ we denote the \emph{fundamental chamber} 
(or \emph{fundamental alcove}), \emph{i.e.}
\begin{align*}
	C_0
	&=\bigl\{x\in\mathfrakA:\sprod{x}{\alpha_0}\leqslant1\text{ and }\sprod{x}{\alpha_i}\geqslant0\text{ for all }
	i \in I_0\big\} \\
	&=\biggl(\bigcap_{i \in I_0}H^+_{\alpha_i;0}\biggr)\cap H^-_{\alpha_0;1}
\end{align*}
which is the fundamental domain for the action of $W^a$ on $\mathfrakA$. Moreover, the group $W^a$ acts simply transitively 
on $C(\Sigma_\Phi)$. This allows us to introduce a chamber system in $C(\Sigma_\Phi)$: For two chambers $C$ and $C'$ and
$i \in I$, we set $C \sim_i C'$ if and only if $C = C'$ or there is $w \in W^a$ such that $C = w . C_0$ and
$C' = w r_i . C_0$. 

The vertices of $C_0$ are $\{0, \lambda_1/m_1, \ldots, \lambda_r/m_r\}$. Let us denote the set of vertices of all 
$C \in C(\Sigma_\Phi)$ by $V(\Sigma_\Phi)$. Under the action of $W^a$, the set $V(\Sigma_\Phi)$ is made up of $r+1$
orbits $W^a.0$ and $W^a.(\lambda_i/m_i)$ for all $i \in I_0$. Thus setting $\tau_{\Sigma_\Phi}(0) = 0$, and
$\tau_{\Sigma_\Phi}(\lambda_i/m_i) = i$ for $i \in I_0$, we obtain the unique labeling
$\tau_{\Sigma_\Phi} : V(\Sigma_\Phi) \rightarrow I$ such that any chamber $C \in C(\Sigma_\Phi)$ has one vertex with each label. 

For each simplicial automorphism $\vphi: \Sigma_\Phi \rightarrow \Sigma_\Phi$ there is a permutation $\pi$ of the set $I$
such that for all chambers $C$ and $C'$, we have $C \sim_i C'$ if and only if $\vphi(C) \sim_{\pi(i)} \vphi(C')$, and 
\[
	\tau_{\Sigma_\Phi}(\vphi(v)) = \pi(\tau_{\Sigma_\Phi}(v))
	\quad\text{ for all } v \in V(\Sigma_\Phi).
\]
A vertex $v$ is called \emph{special} if for each $\alpha \in \Phi^+$ there is $k$ such that $v$ belongs to $H_{\alpha; k}$.
The set of all special vertices is denoted by $V_s(\Sigma_\Phi)$. 

Given $\lambda \in P$ and $w \in W^a$, the set $S = \lambda + w . S_0$ is called a \emph{sector} in $\Sigma_\Phi$ with a
\emph{base vertex} $\lambda$.

Moreover, by \cite[Corollary 3.20]{Abramenko2008}, an affine Coxeter complex $\Sigma_\Phi$ uniquely determines the affine Weyl
group $W^a$ but not a finite root system $\Phi$. In fact, the root systems $\text{C}_r$ and $\text{BC}_r$ have the same
affine Weyl group.

\subsection{Affine buildings}
\label{ss - affine buildings} 
A building $\scrX$ of type $\Sigma$ is called an \emph{affine building} if $\Sigma$ is a Coxeter complex corresponding
to an affine Weyl group. Select a chamber $c_0$ in $C(\scrX)$ and an apartment $\scrA_0$ containing $c_0$. Using
an isomorphism $\psi_0: \scrA_0 \rightarrow \Sigma$ such that $\psi_0(c_0) = C_0$, we define the labeling in $\scrA_0$ by
\[
	\tau_{\scrA_0}(v) = \tau_\Sigma(\psi_0(v)) \quad \text{ for all } v \in V(\scrA_0).
\]
Now, thanks to the building axioms the labeling can be uniquely extended to $\tau: V(\scrX) \rightarrow I$. To turn 
$C(\scrX)$ into a chamber system over $I$ we declare that two chambers $c$ and $c'$ are $i$-adjacent if they share all
vertices except the one of type $i$ (equivalently, they intersect along an $i$-panel). For each $c \in C(\scrX)$ and $i \in I$,
we define
\[
	q_i(c) = \big|\big\{c' \in C(\scrX) : c' \sim_i c \big\}\big| - 1.
\]
In all the paper, we \emph{assume} that $q_i(c)$ only depends on $i$, \emph{i.e.} that $q_i(c)$ is independent of $c$, and therefore the building $\scrX$ is \emph{regular}; we henceforth write $q_i$ instead of $q_i(c)$. 
We also assume that $1 < q_i(c) < \infty$ and therefore the building $\scrX$ is \emph{thick} and \emph{locally finite}. 
A vertex of $\scrX$ is special if it is special in any of its apartments. The set of special vertices is denoted by $V_s$. We
choose the finite root system $\Phi$ in such a way that $\Sigma$ is its Coxeter complex. In all cases except when the
affine group has type $\text{C}_r$ or $\text{BC}_r$, the choice is unique. In the remaining cases we select $\text{C}_r$
if $q_0 = q_r$, otherwise we take $\text{BC}_r$. This guarantees that $q_{\tau(\lambda)} = q_{\tau(\lambda+\lambda')}$
for all $\lambda, \lambda' \in P$, see the discussion in \cite[Section 2.13]{mz1}.
In this article all buildings have reduced type. 

Given two special vertices $x, y \in V_s$, let $\scrA$ be an apartment containing $x$ and $y$ and let $\psi:
\scrA \rightarrow \Sigma$ be a type-rotating isomorphism such that $\psi(x) = 0$ and $\psi(y) \in S_0$, see
\cite[Definition 4.1.1]{park}. We set $\sigma(x, y) = \psi(y) \in P^+$. For $\lambda \in P^+$ and $x \in V_s$, we denote by 
$V_\lambda(x)$ the set of all special vertices $y \in V_s$ such that $\sigma(x, y) = \lambda$. The building axioms entail 
that the cardinality of $V_\lambda(x)$ depends only on $\lambda$, see \cite[Proposition 1.5]{park2}. Let $N_\lambda$ be the 
common value. 

We fix once and for all an origin $o$ which is a special vertex of  the chamber $c_0$.
Let us define a multiplicative function on $\mathfrakA$ by
\[
	\chi_0(\lambda) = \prod_{\alpha \in \Phi^+} q_\alpha^{\sprod{\lambda}{\alpha}}
	\quad \text{ with } \lambda \in \mathfrakA
\]
where $q_\alpha = q_i$ whenever $\alpha \in W\!.\alpha_i$ for $i \in I_0$.
For $w \in W^a$ having the reduced expression $w = r_{i_1} \cdots r_{i_k}$, we set $q_{w} = q_{i_1} \cdots q_{i_k}$. Then
\begin{equation}
	\label{eq:1:1}
	N_\lambda=\frac{W(q^{-1})}{W_\lambda(q^{-1})}\chi_0(\lambda)
\end{equation}
where $W_\lambda=\bigl\{w\in W:w .\lambda=\lambda\bigr\}$, and for any subset $U\subseteq W$ we have set 
\[
	U(q^{-1})=\sum_{w\in U}q_w^{-1}.
\]

\subsection{Spherical harmonic analysis}
\label{ss - analysis on buildings}

In this subsection we summarize spherical harmonic analysis on affine buildings (see \cite{macdo0, park2}).

We consider the averaging operators
\[
	A_\lambda f(x)=\frac1{N_\lambda}\sum_{y\in V_\lambda(x)}f(y) \quad \text{ for all } \lambda\in P^+
\]
acting on functions $f:V_s\to\CC$. Let $\mathscr{A}_0=\cspan\{A_\lambda:\lambda\in P^+\}$ be a commutative unital involutive
algebra (\footnote{The involution is induced by $A_\lambda^\star=A_{\lambda^\star}$ where $\lambda^\star=-w_0.\lambda$.})
whose characters can be expressed in terms of Macdonald functions (see \cite{park2})
\begin{equation}
	\label{eq:1:2}
	P_\lambda(z)=\frac{\chi_0^{-\frac12}(\lambda)}{W(q^{-1})}\sum_{w\in W}\bfc(w.z) e^{\sprod{w.z}{\lambda}}
	\quad \text{ with } \lambda\in P^+ \text{ and } z \in \mathfrak{a}_\CC
\end{equation}
where
\[
	\bfc(z)=\prod_{\alpha\in\Phi^+}\frac{1-q_\alpha^{-1}e^{-\sprod{z}{\alpha\spcheck}}}{1-e^{-\sprod{z}{\alpha\spcheck}}}
\]
is the so-called $\bfc$-function; the values of $P_\lambda$ where the denominator of the $\bfc$-function vanishes 
can be obtained by taking appropriate limits. Specifically, every multiplicative functional on $\mathscr{A}_0$ is given by 
the evaluation
\[
	h_z(A_\lambda)=P_\lambda(z)
	\quad \text{ for all } \lambda\in P^+,
\]
at certain $z\in\mathfrakA_\mathbb{C}$. Moreover, $h_z=h_{z'}$ if and only if $W\!.z+i2\pi L=W\!.z'+i2\pi L$
(see \cite[Theorem 3.3.12(ii)]{macdo0}).

Recall that Macdonald functions as given by the formula \eqref{eq:1:2} are linear combinations
\begin{equation}
	\label{eq:1:5}
	P_\lambda=\sum_{\mu\preceq\lambda}c_{\lambda,\mu} m_{\mu}
\end{equation}
(with nonnegative coefficients) of monomial symmetric functions
\[
	m_\mu=\sum_{\nu\in W\!.\mu}e^{\nu} \quad \text{ for all } \mu\in P^+.
\]
Conversely, every $m_\lambda$ is a linear combination of $P_\mu$'s with $\mu\preceq\lambda$.

Let $\mathscr{A}_2$ be the closure of $\mathscr{A}_0$ in the operator norm on $\ell^2(V_s)$. Then $\mathscr{A}_2$ is a
commutative unital $C^\star$-algebra whose spectrum $\widehat{\mathscr{A}_2}$ equals $i\mathfrakA$ modulo
$W\ltimes i2\pi L$ (\footnote{Specifically, $h_z$ extends to a bounded character on $\mathscr{A}_2$ if and only if
$z\in i\mathfrakA$.}), and the Gelfand--Fourier transform
\[
	\calF\!A(\theta)=h_{i\theta}(A)\quad \text{ with } \theta\in\mathfrakA
\]
identifies $\mathscr{A}_2$ with $\mathcal{C}(\widehat{\mathscr{A}_2})$ the $C^\star$-algebra of (bounded) continuous functions
on $\widehat{\mathscr{A}_2}$. Finally, the following inversion formula holds: for every $A\in\mathscr{A}_0$ and $x,y\in V_s$;
\begin{equation}
	\label{eq:1:3}
	(A\delta_x)(y)=\bigg(\frac1{2\pi}\bigg)^r\frac{W(q^{-1})}{|W|}
	\int_{\mathscr{D}}\calF\!A(\theta) \overline{P_\lambda(i\theta)} \frac{{\rm d} \theta}{|\bfc(i\theta)|^2}
\end{equation}
where $\delta_x$ denotes the Dirac measure at $x$, $\lambda=\sigma(x,y)$ and
\begin{equation}
	\label{FundamentalDomainD}
	\mathscr{D}=\bigl\{\theta\in\mathfrakA:\sprod{\theta}{\alpha\spcheck}\leqslant2\pi\text{ for all }\alpha\in\Phi\bigr\},
\end{equation}
is a fundamental domain for the lattice $2\pi L$ in $\mathfrakA$, see \cite{park2}.

Actually we need another completion of $\mathscr{A}_0$: As in \cite[Section 6]{CartwrightWoess2004}, we consider the following
commutative unital involutive Banach algebra:
\[
	\mathscr{A}_1=\Bigl\{A=\sum\nolimits_{\lambda\in P^+}c_\lambda A_\lambda:\|A\|_1=\sum\nolimits_{\lambda\in P^+}|c_\lambda|<
	\infty\Bigr\}.
\]
Notice that on the one hand the character $h_z$ extends continuously to $\mathscr{A}_1$ if and only if
$\sup_{\lambda\in P^+}|P_\lambda(z)|<\infty$; according to \cite[Theorem 4.7.1]{macdo0} this happens if and only if $\Re z$
belongs to the convex hull of $W\!.\log\chi_0^{\frac12}$. On the other hand $h_z$ is Hermitian if and only if
$P_{\lambda^\star}(z)=\overline{P_\lambda(z)}$ for every $\lambda\in P^+$; as $P_{\lambda^\star}(z)=P_\lambda(-z)$ and
$\overline{P_\lambda(z)}=P_\lambda(\overline{z})$, we deduce that $-\overline{z}\in W\!.z+i2\pi L$. In summary, the continuous
Hermitian characters of $\mathscr{A}_1$ are parametrized by the set
\begin{equation}
	\label{eq:1:4}
	\Sigma=\bigl\{z\in\co(W\!.\log\chi_0^{\frac12})+i\mathfrakA:-\overline{z}\in W\!.z+i2\pi L\bigr\}
\end{equation}
modulo $W\ltimes i2\pi L$ (\footnote{Notice that we may therefore restrict to
$z\in\co(W\!.\log\chi_0^{\frac12})\cap\overline{\mathfrakA_+}+i\mathscr{D}$.}), which we call the \emph{spectrum} of
$\mathscr{X}$. Finally, the Gelfand--Fourier transform extends to a continuous homomorphism from $\mathscr{A}_1$ into
$\mathcal{C}(\Sigma)$.

\section{Multitemporal wave equation on affine buildings}
\label{s - multitemporal wave equation}

In \cite{Semenov1976} Semenov-Tian-Shansky introduced the following 
multitemporal wave equation on Riemannian symmetric spaces $X=G/K$ of non-compact type:
\begin{equation}
	\label{MultitemporalWaveEquationSymmetricSpace}
	\tilde{D}\bigl(\tfrac\partial{\partial t}\bigr)u(t,x)=D_xu(t,x)
	\quad\text{ for all }D\in\mathbb{D}, t \in \mathfrakA\text{ and }x\in X,
\end{equation}
where $\tilde{D}$ is the image of $D$ under the Harish-Chandra isomorphism between the algebra $\mathbb{D}$ of $G$-invariant 
differential operators on $X$, and the algebra $\mathscr{P}(\mathfrakA)^W$ of $W$-invariant polynomials on the Cartan subspace 
$\mathfrakA$. 
This equation was further studied in \cite{PhillipsShahshahani1993}, \cite{Helgason1998},
\cite[Ch.~V, \S\S~5.7--5.11]{Helgason2008}; see  the latter reference for more details.

In this section we introduce a discrete analog of \eqref{MultitemporalWaveEquationSymmetricSpace} on affine buildings $\scrX$
of reduced type, and produce a particular non-trivial solution enjoying two basic properties: finite propagation speed and
uniform space-time exponential decay. Specifically, the algebra $\mathbb{D}$ of invariant differential operators on $X$ is
replaced by the algebra $\mathscr{A}_0$ of averaging operators on $\mathscr{X}$ and the Harish-Chandra isomorphism $\Gamma$
by the isomorphism induced by $A_\lambda\mapsto P_\lambda$ ($\lambda\in P^+$) between $\mathscr{A}_0$ and the algebra of
$W$-invariant trigonometric polynomials on $i\mathfrakA$. Thus \eqref{MultitemporalWaveEquationSymmetricSpace} becomes 
\begin{equation}
	\label{eq:2:7}
	\sum_{\atop{\mu\in P^+}{\mu\preceq\lambda}}c_{\lambda,\mu}
	\sum_{\omega\in W\!.\mu}u(\nu+\omega,x)=A_\lambda u(\nu,x)
	\quad\text{ for all }\lambda,\nu\in P^+\text{ and }x\in V_s
\end{equation}
where the coefficients $c_{\lambda,\mu}$ are defined by \eqref{eq:1:5}.

Next, let us consider the function $u:P^+\times V_s\rightarrow\CC$ whose partial Gelfand--Fourier transform with respect to
the second variable is given by
\[
	u(\nu,\hat{\theta})=\sum_{w\in W}e^{i\langle w.\theta,\nu\rangle}=|W_\nu| m_\nu(i\theta)
	\quad\text{ for all }\nu\in P^+\text{ and }\theta\in\mathfrakA,
\]
where $W_\nu = \{w \in W : w.\nu = \nu \}$. 
In particular, $u(\nu, \cdot)$ belongs to $\mathscr{A}_0$. We claim that $u$ solves \eqref{eq:2:7} which writes
\begin{equation}
	\label{eq:2:8}
	\sum_{\mu\preceq\lambda}c_{\lambda,\mu}\sum_{\omega\in W\!.\mu}u(\nu+\omega,\hat{\theta})
	=P_\lambda(i\theta) u(\nu,\hat{\theta})
	\quad\text{ for all }\nu\in P^+\text{ and }\theta\in\mathfrakA,
\end{equation}
after taking the partial Gelfand--Fourier transform. Indeed,
\begin{align*}
	\sum_{\mu \preceq \lambda} c_{\lambda, \mu} \sum_{\omega \in W\!. \mu} u(\nu + \omega, \hat{\theta} )
	&=
	\sum_{\mu \preceq \lambda} c_{\lambda, \mu} \sum_{\omega \in W\!. \mu} \sum_{w \in W} 
	e^{i\sprod{w . \theta}{\nu + \omega}} \\
	&=
	\sum_{\mu \preceq \lambda} c_{\lambda, \mu} \sum_{w \in W} 
	e^{i\sprod{w.\theta}{\nu}}\sum_{\omega\in W\!.\mu}e^{i\sprod{w.\theta}{\omega}} \\
	&=
	\sum_{\mu \preceq \lambda} c_{\lambda, \mu}m_{\mu}(i \theta) \sum_{w \in W} e^{i \sprod{w . \theta}{\nu}} \\
	&=
	P_\lambda(i\theta)u(\nu, \hat{\theta}).
\end{align*}
If $\sigma(o, x) = \omega$, then by the inversion formula \eqref{eq:1:3} we obtain
\begin{align*}
	u(\mu,x)
	&=\bigg(\frac1{2\pi}\bigg)^r \frac{W(q^{-1})}{\abs{W}}
	\int_{\mathscr{D}}u(\mu,\hat{\theta}) \overline{P_\omega(i\theta)} \: \frac{{\rm d}\theta}{\abs{\bfc(i\theta)}^2} \\
	&=\bigg(\frac1{2\pi}\bigg)^r \frac{W(q^{-1})}{\abs{W\!. \mu}}
	\int_{\mathscr{D}}m_{\mu}(i\theta) \overline{P_\omega(i\theta)} \: \frac{{\rm d}\theta}{\abs{\bfc(i\theta)}^2}.
\end{align*}
Now, using the explicit formula for Macdonald functions \eqref{eq:1:2}, we get
\begin{align*}
	u(\mu, x)
	&=\chi_0(\omega)^{-\frac12} \frac{1}{\abs{W\!.\mu}}\sum_{w \in W}\bigg(\frac1{2\pi}\bigg)^r
	\int_{\mathscr{D}}m_{\mu}(i\theta) e^{-i\sprod{w.\theta}{\omega}} \: \frac{{\rm d}\theta}{\bfc(iw.\theta)}\\
	&=\chi_0(\omega)^{-\frac12} \underbrace{\color{black}\frac{\abs{W}}{\abs{W\!.\mu}}}_{|W_\mu|} \bigg(\frac1{2\pi}\bigg)^r
	\int_{\mathscr{D}}m_{\mu}(i \theta) e^{-i\sprod{\omega}{\theta}} \: \frac{{\rm d}\theta}{\bfc(i\theta)}.
\end{align*}
The last integral vanishes if  $\omega \not \preceq \mu$; in other words, $u(\mu, \cdot)$ is supported in
\[
	\big\{x \in V_s : \sigma(o, x) \preceq \mu \big\}.
\]
This property is a higher rank analog of finite propagation speed for the wave equation on homogeneous trees
(see \cite[Section 4]{Anker2012} or \cite[Lemma 2]{BrooksLindenstrauss2010}).

Next, we write
\[
	\frac{1-e^{-\alpha\spcheck} }{1 - q_{\alpha}^{-1} e^{-\alpha\spcheck}} = 
	1 + (1-q_{\alpha}) \sum_{j_\alpha = 1}^\infty q_{\alpha}^{-j_\alpha} 
	e^{-j_\alpha \alpha\spcheck},
\]
thus
\[
	\frac{1}{\bfc} = \sum_{\nu \in Q^+} C(\nu) e^{-\nu}
\]
where for $\nu \in Q^+$ we have set
\[
	C(\nu) = \sum_{
	(j_\alpha : \alpha \in \Phi^+) \in \calP(\nu)}
	\prod_{\stackrel{\alpha \in \Phi^+}{j_\alpha > 0}} (1-q_\alpha) q_\alpha^{-j_\alpha}
\]
with
\[
	\calP(\nu) = 
	\Big\{
	(j_\alpha : \alpha \in \Phi^+) \in {\NN^{|\Phi^+|}}: \sum_{\alpha \in \Phi^+} j_\alpha \alpha\spcheck = \nu
	\Big\}.
\]
Therefore,
\begin{align*}
	u(\mu, x)
	&=\chi_0(\omega)^{-\frac12}\frac{\abs{W}}{\abs{W\!.\mu}} \bigg(\frac1{2\pi}\bigg)^r
	\sum_{\nu\in Q^+}C(\nu)\int_{\mathscr{D}}m_{\mu}(i\theta) e^{-i\sprod{\omega+\nu}{\theta}}{\: \rm d}\theta
	\nonumber\\
	&=\chi_0(\omega)^{-\frac12}\frac{\abs{W}}{\abs{W\!.\mu}}
	\sum_{\nu\in Q^+}\frac1{\abs{W\!.(\omega+\nu)}}C(\nu)\sprod{m_{\mu}}{m_{\omega+\nu}}.
\end{align*}
Since
\[
	\sprod{m_{\mu}}{m_{\omega+\nu}} = 
	\begin{cases}
		|W.\mu| & \text{if } \omega + \nu \in W . \mu \\
		0 & \text{otherwise,}
	\end{cases}
\]
we obtain
\begin{equation}
	u(\mu, x) = \chi_0(\omega)^{-\frac12}\frac{\abs{W}}{\abs{W\!.\mu}} 
	\sum_{\atop{w \in W}{\omega\preceq w.\mu}} C(w.\mu-\omega). 
	\label{eq:2:9}
\end{equation}
Our next aim is to estimate 
\[
	\chi_0(\omega)^{-\frac12}\sum_{\atop{w\in W}{\omega\preceq w.\mu}}|C(w.\mu-\omega)|.
\]
Let us recall that the condition $w.\mu\succeq\omega$ implies that $w.\mu-\omega$ belongs to the cone $Q^+$ generated by the 
positive co-roots.

\begin{lemma}
	\label{lem:2:1}
	There are $C > 0$ and $1 < A < \min\{q_\alpha : \alpha \in \Phi^+\}$, such that for each $\nu \in Q^+$,
	\[
		\abs{C(\nu)} \leqslant C A^{-\sprod{\nu}{\tilde{\rho}}}
	\]
	where 
	\[
		\tilde{\rho} = \frac12 \sum_{\alpha \in \Phi^{+}} \alpha.
	\]
\end{lemma}
\begin{proof}
	Let
	\[
		K = \max\big\{\sprod{\alpha\spcheck}{\tilde{\rho}} : \alpha \in \Phi^+ \big\}.
	\]
	Since $K \geqslant 1$, we have
	\[
		\sum_{\alpha \in \Phi^+} j_\alpha \geqslant \frac{1}{K}
		\sum_{\alpha \in \Phi^+} j_\alpha \sprod{\alpha\spcheck}{\tilde{\rho}}
		=
		\frac{1}{K} \sprod{\nu}{\tilde{\rho}}.
	\]
	Hence, (\footnote{We write $f \lesssim g$ if there is $C > 0$ such that $f \leqslant C g$.})
	\begin{align*}
		|C(\nu)|
		&\leqslant \sum_{(j_\alpha) \in \calP(\nu)} \prod_{\alpha : j_\alpha > 0} (q_\alpha-1) q_{\alpha}^{-j_\alpha} \\
		&\lesssim 
		\sum_{k \geqslant \frac{1}{K} \sprod{\nu}{\tilde{\rho}}} 
		\Big|\Big\{(j_\alpha) \in \calP(\nu) : \sum_{\alpha \in \Phi^+} j_\alpha = k \Big\}\Big| A^{-2 K k}
	\end{align*}
	where $A^{2K} = \min\{q_\alpha : \alpha \in \Phi^+\}$. Consequently,

	\begin{align*}
		|C(\nu)| 
		&\lesssim\sum_{k \geqslant \frac{1}{K} \sprod{\nu}{\tilde{\rho}}} k^{|\Phi^+|} A^{-2 K k} \\
		&\lesssim A^{-\sprod{\nu}{\tilde{\rho}}}
	\end{align*}
	and the lemma follows.
\end{proof}

\begin{lemma}
	\label{lem:2:2}
	There are $C > 0$ and $c > 0$ such that for all $\mu, \omega \in P^+$,
	\[
		\chi_0(\omega)^{-\frac12} \sum_{\stackrel{w \in W}{w .\mu \succeq \omega}} |C(w . \mu - \omega)|
		\leqslant
		C e^{-c \norm{\mu}}.
	\]
\end{lemma}
\begin{proof}
	Let $A$ be the constant from Lemma \ref{lem:2:1}. Since $1 < A < \min\{q_\alpha:\alpha\in\Phi^+\}$, we have
	\begin{align*}
		\chi_0(\omega)^{-\frac12}\leqslant\prod_{\alpha\in{\Phi^+}} q_\alpha^{-\frac12\sprod{\omega}{\alpha}}
		\leqslant A^{-\sprod{\omega}{\tilde{\rho}}}.
	\end{align*}
	Hence, by Lemma \ref{lem:2:1}, we can write
	\begin{align*}
		\sum_{\stackrel{w\in W} {w.\mu\succeq\omega}}\chi_0(\omega)^{-\frac12}|C(w.\mu-\omega)|
		&\lesssim\sum_{\stackrel{w\in W}{w.\mu\succeq\omega}}A^{-\sprod{w.\mu-\omega}{\tilde{\rho}}}
		A^{-\sprod{\omega}{\tilde{\rho}}}.
	\end{align*}
	Since $w.\mu\succeq\omega$ and $\omega\in P^+$, we have
	\[
		\sprod{w.\mu}{\tilde{\rho}}\gtrsim\norm{w .\mu}=\norm{\mu},
	\]
	thus there is $c > 0$ such that 
	\begin{align*}
		\sum_{\stackrel{w \in W} {w .\mu\succeq\omega}}\chi_0(\omega)^{-\frac12} |C(w .\mu-\omega)|
		\lesssim e^{-c \norm{\mu}}
	\end{align*}
	which completes the proof.
\end{proof}
Using \eqref{eq:2:9} and Lemma \ref{lem:2:2}, we can easily deduce the following consequence.
\begin{theorem}
	\label{thm:2:3}
	Let $u$ be the real function on $P^+\times V_s$ with Fourier--Gelfand transform satisfying
	\[
		\calF\!u(\mu,\theta)=|W_\mu| m_\mu(i\theta)\quad\text{ for all }\mu\in P^+\text{ and }\theta\in\mathfrakA.
	\]
	Then for each $\mu\in P^+$, the function $u(\mu,\cdot)$ is supported in $\bigl\{x\in V_s:\sigma(o, x)\preceq\mu\bigr\}$.
	Moreover, there are $C>0$ and $c > 0$ such that
	\begin{equation}
		\label{eq:4}
		|u(\mu, x)|\leqslant Ce^{-c\norm{\mu}}\quad\text{ for all }\mu\in P^+\text{ and }x\in V_s.
	\end{equation}
\end{theorem}

\begin{remark}
	In the case of $\tilde{A}_2$ buildings we can obtain sharper estimates than \eqref{eq:4}. Namely, there is $C > 0$
	such that for all $\mu \in P^+$ and $x \in V_s$,
	\[
		|u(\mu, x)| \leqslant C q^{-\max\{\sprod{\alpha_1}{\mu}, \sprod{\alpha_2}{\mu}\}}
	\]
	where $q+1$ denotes the thickness of the building. 
\end{remark}

\begin{remark}
	It would be interesting to extend Theorem \ref{thm:2:3} to all fundamental solutions of \eqref{eq:2:7}.
	We intend to return to this question in the future.
\end{remark}

\section{Kernel construction}
\label{s - kernel construction}
In this section we prove our main result stated in the introduction as Theorem \ref{th - kernel intro} which holds for any 
(regular, thick, locally finite) affine building of reduced type. The proof, inspired by the kernel construction of Brooks 
and Lindenstrauss in \cite[Section 2]{BrooksLindenstrauss2010}, \cite[Section 3]{BrooksLindenstrauss2013} and 
\cite[Section 3.1]{BrooksLindenstrauss2014}, uses the solution of the multitemporal wave equation described in Theorem 
\ref{thm:2:3}.

\begin{theorem}
	\label{thm:3:1}
	Let $z_0=\zeta_0+i\theta_0\in\Sigma$. There are $M(\zeta_0), N(\zeta_0)\in\NN$ such that, for all integers
	$M\geqslant M(\zeta_0)$ and $N\geqslant N(\zeta_0)$, there is a function $k: V_s \times V_s \rightarrow \CC$ such that
	\begin{enumerate}[start=1, label=\rm (\roman*), ref=\roman*]
		\item
		\label{en:2:1}
		for all $x, y \in V_s$ the value of $k(x, y)$ depends only on the vectorial distance $\sigma(x, y)$; 
		moreover, $k(x, y)$ vanishes when
		\[
			\sigma(x, y)\not\preceq16M\bigl(8\hspace{1pt}\pi hMN\bigr)^r\rho
		\]
		where
		\[
			h=\sum_{\alpha\in\Phi^+}\height(\alpha\spcheck)
		\]
		and $\height(\alpha\spcheck)$ denotes the height of the coroot $\alpha\spcheck$ with respect to the coroot basis
		$\big\{\alpha_1\spcheck, \ldots, \alpha_r\spcheck\big\}$;
		\item
		\label{en:2:2}
		there are $C > 0$ and $c > 0$ such that
		\[
			\sup_{x, y \in V_s} |k(x, y)| \leqslant C e^{-cN};
		\]
		\item
		\label{en:2:3}
		for all $x \in V_s$ and $z \in \Sigma$, we have 
		\[
			\calF\!\bigl(k(x,\cdot)\bigr)(z)\geqslant-1\quad\text while\quad\calF\!\bigl(k(x,\cdot)\bigr)(z_0)\geqslant M.
		\]
	\end{enumerate}
\end{theorem}

\begin{proof}
	We are going to define a kernel with desired properties by describing it on the Gelfand--Fourier side.
	In view of Theorem \ref{thm:2:3}, the convenient building blocks are the elementary symmetric functions which essentially
	reduces the proof to showing \eqref{en:2:3}.
	
	\smallskip

	Let $b$ and $M$ be two positive integers whose values will be specified later. Let $K_{M, b}$ be the complex function on 
	$\mathfrakA_\CC$ defined by
	\begin{equation}
		\label{eq:3:10}
		K_{M, b}(\theta)=\frac1{1+4|W|M} \Bigl(\sum_{k=0}^{4M}m_{k\rho}(ib\theta)\Bigr)^2-1.
	\end{equation}
	It is clear that $K_{M, b}$ is invariant under the action of $W$ as well as it is invariant under translations by vectors 
	from $2\pi iL$. Given $z\in\Sigma$, there are $\tilde{w}\in W$ and $\ell\in L$ such that 
	$\tilde{w}.z+i2\pi\ell=-\overline{z}$. We claim that for all $k\in\NN$,
	\begin{equation}
		\label{eq:3:11}
		m_{k \rho}(z) \in \RR.
	\end{equation}
	Indeed, we have
	\begin{align*}
		\overline{m_{k \rho}(z)} 
		&=\sum_{w\in W}e^{k\sprod{\overline{z}}{w.\rho}} \\
		&=\sum_{w\in W}e^{k\sprod{-\tilde{w}.z-i2\pi\ell}{w.\rho}} \\
		&=\sum_{w\in W}e^{k\sprod{z}{\tilde{w}^{-1}ww_0.\rho}}
		=m_{k\rho}(z)
	\end{align*}
	where $w_0$ is the longest element of $W$ and $w_0 .\rho=-\rho$. Using \eqref{eq:3:11} we easily deduce that
	$K_{M, b}\geqslant-1$ on $\Sigma$. 
	
	Our next step is to show that there are $M(\zeta_0), N(\zeta_0) \in \NN$ such that for all integers $M \geqslant M(\zeta_0)$
	and $N\geqslant N(\zeta_0)$ there is an integer $b\geqslant N$ such that
	\begin{equation}
		\label{eq:3:12}
		K_{M,b}(z_0)\geqslant M.
	\end{equation}
	First, let us see how \eqref{eq:3:12} allows us to complete the proof. Using the inversion formula \eqref{eq:1:3}, we define 
	the kernel $k:V_s \times V_s \rightarrow\CC$ by
	\[
		k(x, y)=\biggl(\frac1{2\pi}\biggr)^r\frac{W(q^{-1})}{\abs{W}} \int_{\mathscr{D}}
		K_{M, b}(\theta) \overline{P_{\omega}(i\theta)} \: \frac{{\rm d}\theta}{\abs{\bfc(i\theta)}^2}
	\]
	where $\omega=\sigma(x, y)$. It follows from \eqref{eq:3:10} and \eqref{eq:3:12} that $k(x, y)$ satisfies \eqref{en:2:3}. 
	Next, we deduce from \eqref{eq:3:10} that
	\[
		\int_{\mathscr{D}}K_{M, b}(\theta) \overline{P_{\omega}(i\theta)} \: \frac{{\rm d} \theta}{\abs{\bfc(i\theta)}^2}
	\]
	vanishes whenever $\omega\not\preceq8Mb\rho$. Hence, $k(x, y)\neq0$ implies that $\sigma(x,y)\preceq8Mb\rho$,
	which yields \eqref{en:2:1} because $b\leqslant2B^r$. Thus it remains to prove \eqref{en:2:2}. 
	Observe that, according to Corollary \ref{thm:2:3}, there are $C > 0$ and $c > 0$ such that for all $\mu \in P^+$,
	\begin{equation}
		\label{eq:3:25}
		\|\calF^{-1}m_{\mu}\|_\infty\leqslant Ce^{-c \norm{\mu}}.
	\end{equation}
	We claim that
	\begin{equation}
		\label{eq:3:23}
		K_{M, b}(\theta)=\sum_{\atop{\nu\in P^+}{0\prec\nu\preceq 8M\rho}} c_\nu m_{\nu}(ib \theta)
 	\end{equation}
	with $0 \leqslant c_{\nu} \leqslant 1$. Consider the Euclidean Fourier transform between $\mathfrakA/2\pi L$ and $P$.
	Then $m_\nu(ib\theta)$ corresponds to the characteristic function of $bW.\nu$,
	\[
		\sum_{k=0}^{4M}m_{k\rho}(ib\theta)
	\]
	to the characteristic function of the set
	\[
		X=\bigsqcup_{k=0}^{4M}\bigl(kbW.\rho\bigr),
	\]
	whose cardinality $|X|$ is equal to $1+4|W|M$, and $K_{M,b}(\theta)$ to
	\[
		\frac1{|X|}\ind{X}*\ind{X}-\ind{\{0\}}.
	\]
	As $X$ is invariant under $W$ and $-\text{Id}$, we have
	\begin{equation}
		\label{eq:3:23bis}
		\ind{X}*\ind{X}
		= |X| \ind{\{0\}}+\sum_{\substack{\nu\in P^+\\0\prec\nu\preceq 8M\rho}} c_\nu'\ind{bW.\nu}\,,
	\end{equation}
	where $c_\nu'=|(b\nu+X) \cap X | \in [0, |X| ]$. In conclusion, \eqref{eq:3:23bis} corresponds to \eqref{eq:3:23} 
	via the Euclidean Fourier transform, with $c_\nu=c_\nu'/ |X| \in[0,1]$.

	Lastly, by \eqref{eq:3:23} and \eqref{eq:3:25}, we get
	\begin{align*}
		\sup_{x, y \in V_g} \abs{k(x, y)}
		&\leqslant \sum_{\stackrel{\nu \in P^+}{0 \prec \nu \preceq 8 M \rho}}
		\|\calF^{-1} m_{b \nu}\|_\infty \\
		&\leqslant
		\sum_{\stackrel{\nu \in P^+}{0 \prec \nu \preceq 8 M \rho}}
		e^{-c b |\nu|} \\
		&\leqslant
		\sum_{j = 1}^{8 M \norm{\rho}}
		e^{-c b j} \# \big\{\nu \in P^+ : |\nu| = j \big\} \\
		&\leqslant
		C'
		\sum_{j = 1}^{8 M \norm{\rho}}
		j^{r-1} e^{-c b j}
		=
		\calO(e^{-c b})
	\end{align*}
	which is $\calO(e^{-c N})$. This proves \eqref{en:2:2}. 

	\smallskip
	It remains to show \eqref{eq:3:12}. The proof is divided into three cases.

	\noindent
	\emph{Case 1:} If $\theta_0 = 0$, then for all $k \in \NN$,
	\begin{align*}
		m_{k \rho}(b \zeta_0) 
		&= \sum_{w \in W} e^{k b \sprod{\zeta_0}{w . \rho}} \\
		&= \frac12 \sum_{w \in W} e^{k b \sprod{\zeta_0}{w . \rho}}  
		+ \frac12 \sum_{w \in W} e^{k b \sprod{\zeta_0}{ww_0 . \rho}} \\
		&= \sum_{w \in W} \cosh\big(k b \sprod{\zeta_0}{w . \rho}\big) \\
		&\geqslant |W|.
	\end{align*}
	Thus
	\begin{equation}
		\label{eq:3:13}
		K_{M, b}(\zeta_0)
		\geqslant \frac{1}{1 + 4 |W|M} \Big(\sum_{k = 0}^{4M} 
		m_{k\rho}(b \zeta_0 )\Big)^2 - 1
		\geqslant
		4 M |W|
	\end{equation}
	for every $b, M \in \NN$. In this case we set $b = N$.

	\smallskip	 

	\noindent
	\emph{Case 2:} Assume that $\zeta_0 = 0$. We start by the following observation: if
	\[
		0 < |x| \leqslant \frac{\pi}{8 M+1},
	\]
	then
	\begin{align*}
		D_{4M}(i x) = 
		\sum_{k=-4M}^{4M} e^{i k x} &= 
		\frac{\sin(4M+\tfrac12)x}{\sin\tfrac12x} \\
		&= (8M+1)\frac{\sin(4M+\tfrac12)x}{(4 M+\tfrac{1}{2})x} \frac{\tfrac{1}{2}x}{\sin\frac12x} \\
		&\geqslant\frac2\pi\bigl(8M+1\bigr).
	\end{align*}
	If $x = 0$ then
	\[
		D_{4M}(i x)=\sum_{k=-4M}^{4M}e^{ikx}=8M+1.
	\]
	Therefore
	\[
		D_{4M}(ix)
		\geqslant
		\frac2\pi(8M+1)
		\quad\text{ for }x \in\left[-\frac\pi{8M+1},\frac\pi{8M+1}\right].
	\]
	Using the last inequality, we obtain
	\begin{align}
		\nonumber
		\sum_{k = 0}^{4M} m_{k\rho}(i \theta) 
		&= 1 - \frac12 |W| + \frac12 \sum_{w \in W} D_{4M}(i\sprod{\theta}{w . \rho}) \\
		\nonumber
		&\geqslant
		1 - \frac12 |W| + \frac{8}{\pi}|W| M + \frac{1}{\pi} |W| \\
		\label{eq:3:14}
		&\geqslant
		1 + 2 |W|M
	\end{align}
	provided that
	\[
		\sprod{\theta}{w . \rho} 
		\in \Big[-\frac{\pi}{8 M + 1}, \frac{\pi}{8M+1}\Big]
		\quad \text{ for all } w \in W.
	\]
	Next, we are going to find $b$ such that
	\begin{equation}
		\label{eq:3:24}
		b \sprod{\theta_0}{w . \rho} \in \Big[-\frac{\pi}{8 M + 1}, \frac{\pi}{8M+1} \Big] + 2 \pi \ZZ
		\quad \text{ for all } w \in W.
	\end{equation}
	Let
	\begin{equation}
		\label{eq:3:15}
		B = 2 h(8M+1) N. 
	\end{equation}
	By Dirichlet's approximation theorem, there is $b' \in \{1, 2, \ldots, B^r\}$ such that
	\[
		b' \sprod{\theta_0}{\alpha_j\spcheck} \in \Big[0,\frac{2\pi}{B}\Big) + 2 \pi \ZZ
		\quad\text{ for all } j \in \{1, 2, \ldots, r\}.
	\]
	Hence, 
	\[
		b'\sprod{\theta_0}{\alpha\spcheck}
		\in
		\Big[0,\height(\alpha\spcheck)\frac{2\pi}{B}\Big)+2\pi\ZZ
		\quad\text{ for all }\alpha\in \Phi^{+},
	\]
	and so
	\[
		2 b' \sprod{\theta_0}{w . \rho} 
		\in 
		\Big(-\frac{2\pi h}{B},  \frac{2\pi h}{B} \Big)
		+ 2\pi \ZZ
		\quad\text{ for all } w \in W.
	\]
	Let $b = 2 b' \lceil N/ (2b') \rceil$. Then $2 N \leqslant b \leqslant 2 b' N$. Moreover, if $N \leqslant 2 b'$ then
	\[
		b = 2b' \leqslant 2 B^r.
	\]
	Otherwise $2 b' < N$ and
	\[
		b \leqslant N + 2 b' < 2 N \leqslant B \leqslant 2 B^r.
	\]
	Hence, we get \eqref{eq:3:24}. Now, in view of \eqref{eq:3:14},
	\begin{align*}
		K_{M, b}(\theta_0)
		&=
		\frac{1}{1 + 4 |W|M} \bigg(\sum_{k=0}^{4M} m_{k\rho}(i b \theta_0)\bigg)^2 - 1 \\
		&\geqslant
		\frac{1}{1 + 4 |W|M} \Big(1 + 2 |W| M\Big)^2 - 1 \\
		&\geqslant
		\frac12 |W| M \geqslant M
	\end{align*}
	for every $M \in \NN$.
	\smallskip 

	\noindent
	\emph{Case 3:} It remains to consider $z_0=\zeta_0+i\theta_0\in\Sigma$ with
	$\zeta_0\in\overline{\mathfrakA_+}\smallsetminus\{0\}$ 
	and $\theta_0\in\mathfrakA$. In fact this case requires more work. First, let us observe that $\sprod{\zeta_0}{\rho}>0$. 
	As $z_0\in\Sigma$, there are $w_1\in W$ and $\ell\in L$ such that $w_1.\zeta_0=-\zeta_0$ and 
	$w_1.\theta_0+2\pi\ell=\theta_0$. Let us consider
	\[
		W(\zeta_0) = \big\{w \in W: w . \zeta_0 = \zeta_0 \big\}.
	\]
	We claim the following
	\begin{claim}
		\[
			W(\zeta_0) 
			=
			w_0 W(\zeta_0) w_1 
			= \big\{w \in W : \sprod{w . \zeta_0}{\rho} = \sprod{\zeta_0}{\rho} \big\}.
		\]
	\end{claim}
	To prove the first equality, let $w \in W(\zeta_0)$. Then
	\[
		\sprod{w_0 w w_1. \zeta_0}{\rho} 
		=
		\sprod{w. \zeta_0}{\rho} = \sprod{\zeta_0}{\rho}
	\]
	where we have used $w_1 . \zeta_0 = - \zeta_0$ and $w_0 . \rho = -\rho$.
	To justify the second equality, it is enough to prove the inclusion $\supseteq$. Since 
	$\zeta_0 \in \overline{\mathfrakA_+} \setminus \{0\}$, for each $w\in W$, there are $\gamma_j\geqslant 0$ such that
	\begin{equation}
		\label{eq:3:16}
		\zeta_0 - w.\zeta_0 = \sum_{j = 1}^r \gamma_j \alpha_j.
	\end{equation}
	Thus if $\sprod{w . \zeta_0}{\rho} = \sprod{\zeta_0}{\rho}$, then
	\[
		0=\sprod{\zeta_0-w.\zeta_0}{\rho}=\sum_{j=1}^r\gamma_j,
	\]
	and so $\gamma_1=\ldots=\gamma_r=0$, that is $w\in W(\zeta_0)$, proving the claim 

	\smallskip
	For future use, let us observe that 
	\begin{equation}
		\label{eq:3:17}
		\kappa = \min\big\{\sprod{\zeta_0-w.\zeta_0}{\rho}:w\in W\setminus W(\zeta_0)\big\}>0. 
	\end{equation}
	Since for each $b\in\NN$ and $k\in\ZZ$, we have
	\begin{align*}
		\sum_{w \in W(\zeta_0)}e^{ikb\sprod{\theta_0}{w.\rho}}
		&=\sum_{w\in W(\zeta_0)}e^{ikb\sprod{\theta_0}{w_1^{-1}ww_0.\rho}} \\
		&=\sum_{w \in W(\zeta_0)}e^{-ikb\sprod{\theta_0}{w.\rho}},
	\end{align*}
	we obtain
	\begin{align*}
		\sum_{w \in W(\zeta_0)} e^{k b \sprod{z_0}{w . \rho}} 
		&= 
		e^{kb \sprod{\zeta_0}{\rho}} \sum_{w \in W(\zeta_0)} e^{ikb\sprod{\theta_0}{w . \rho}} \\
		&=
		e^{kb \sprod{\zeta_0}{\rho}}
		\frac12\bigg\{
		\sum_{w \in W(\zeta_0)} e^{ikb\sprod{\theta_0}{w . \rho}} 
		+
		\sum_{w \in W(\zeta_0)} e^{-ikb\sprod{\theta_0}{w . \rho}} \bigg\} \\
		&=
		e^{kb \sprod{\zeta_0}{\rho}} \sum_{w \in W(\zeta_0)} \cos\sprod{\theta_0}{w. \rho}.
	\end{align*}
	Now, using $\cos t=1-2\bigl(\sin{\frac t2}\bigr)^2$, we arrive at
	\begin{align*}
		\sum_{w \in W(\zeta_0)} e^{k b \sprod{z_0}{w . \rho}}
		&=
		|W(\zeta_0)| e^{kb \sprod{\zeta_0}{\rho}}
		-
		2 e^{kb \sprod{\zeta_0}{\rho}}
		\sum_{w \in W(\zeta_0)} \bigg(\sin \frac{kb}{2} \sprod{\theta_0}{w . \rho} \bigg)^2.
	\end{align*}
	In order to estimate $K_{M, b}(z_0)$, we write
	\begin{equation}
		\label{eq:3:18}
		\sum_{k = 0}^{4 M} m_{k \rho}(b z_0)
		=
		1 + S_1 - S_2 + S_3 + S_4 + S_5
	\end{equation}
	where we have set
	\begin{equation}
		\label{eq:3:19}
		S_1 = \abs{W(\zeta_0)}e^{4 M b\sprod{\zeta_0}{\rho}},
	\end{equation}
	and 
	\begin{equation}
		\label{eq:3:20}
		\begin{aligned}
		S_2 &= 2 e^{4Mb\sprod{\zeta_0}{\rho}} \sum_{w  \in W(\zeta_0)} \big(\sin 2 M b \sprod{\theta_0}{w . \rho} \big)^2 \\
		S_3 &= \sum_{k = 1}^{4M-1} e^{kb\sprod{\zeta_0}{\rho}} \sum_{w \in W(\zeta_0)} \cos kb \sprod{\theta_0}{w . \rho} \\
		S_4 &= \sum_{k = 1}^{4M} \sum_{\stackrel{w \in W \setminus W(\zeta_0)}{\sprod{\zeta_0}{w.\rho} > 0}}
		e^{kb \sprod{z_0}{w . \rho}}\\
		S_5 &= \sum_{k = 1}^{4M} \sum_{\stackrel{w \in W \setminus W(\zeta_0)}{\sprod{\zeta_0}{w.\rho} \leqslant 0}}
		e^{kb \sprod{z_0}{w . \rho}}.
		\end{aligned}
	\end{equation}
	We claim that $S_1$ is the leading term. To see this, let $M(\zeta_0)$ and $N(\zeta_0)$ be two positive integers such
	that the following inequalities holds for all $M \geqslant M(\zeta_0)$ and $N \geqslant N(\zeta_0)$,
	\begin{equation}
		\label{eq:3:21}
		\begin{aligned}
			\frac{e^{4Mb\sprod{\zeta_0}{\rho}}}{M} 
			&\geqslant \frac{32 \abs{W}}{\abs{W(\zeta_0}},\\
			e^{4Mb\kappa}&\geqslant\frac{\kappa'}{\abs{W(\zeta_0)}},\\
			e^{N\sprod{\zeta_0}{\rho}}&\geqslant 9
		\end{aligned}
	\end{equation}
	where $\kappa$ is defined in \eqref{eq:3:17}, and
	\[
		\kappa' = \sum_{\stackrel{w \in W \setminus W(\zeta_0)}{\sprod{\zeta_0}{w . \rho} > 0}}
		\frac{1}{1 - e^{-\sprod{\zeta_0}{w.\rho}}} > 0.
	\]
	The estimate for $S_5$ is straightforward:
	\[
		\abs{S_5}\leqslant\sum_{k=1}^{4M}
		\sum_{\stackrel{w\in W\setminus W(\zeta_0)}{\sprod{\zeta_0}{w.\rho}\leqslant 0}}
		e^{kb\sprod{\zeta_0}{w.\rho}}\leqslant4M\abs{W}\leqslant \frac{S_1}8.
	\]
	To bound $S_4$ we proceed as follows:
	\begin{align*}
		\abs{S_4} 
		\leqslant \sum_{\stackrel{w \in W \setminus W(\zeta_0)}{\sprod{\zeta_0}{w.\rho} > 0}}
		\sum_{k=1}^{4M}e^{kb\sprod{\zeta_0}{w.\rho}} 
		&=\sum_{\stackrel{w \in W \setminus W(\zeta_0)}{\sprod{\zeta_0}{w.\rho} > 0}}
		\frac{e^{4Mb \sprod{\zeta_0}{w . \rho}} -1}{1 - e^{-b \sprod{\zeta_0}{w . \rho}}} \\
		&\leqslant 
		e^{4Mb (\sprod{\zeta_0}{\rho} - \kappa)} \kappa' \leqslant\frac{S_1}8.
	\end{align*}
	For $S_3$ we have
	\begin{align*}
		\abs{S_3} 
		\leqslant 
		\abs{W(\zeta_0)} \sum_{k = 1}^{4M-1} e^{kb \sprod{\zeta_0}{\rho}} 
		&=\abs{W(\zeta_0)}
		\frac{e^{4Mb \sprod{\zeta_0}{\rho}}-e^{b\sprod{\zeta_0}{\rho}}}{e^{b\sprod{\zeta_0}{\rho}}-1} \\
		&\leqslant{\frac1{e^{b\sprod{\zeta_0}{\rho}}-1}} S_1,
	\end{align*}
	which is smaller than $\frac{S_1}8$ if $b\geqslant N$.

	To control $S_2$, we show that there is a positive integer $b \geqslant N$, such that
	\begin{equation}
		\label{eq:3:22}
		\sum_{w \in W(\zeta_0)} \big(\sin 2 Mb \sprod{\theta_0}{w . \rho} \big)^2 
		\leqslant \frac{\abs{W(\zeta_0)}}{16}.
	\end{equation}
	For this purpose we resume the arguments in Case 2 based on the Dirichlet approximation theorem. Let 
	\[
		B = 8 \pi h M N.
	\]
	Then there is $1 \leqslant b' \leqslant B^r$, such that
	\[
		b' \sprod{\theta_0}{\spcheck{\alpha_j}} 
		\in
		\Big[0, \frac{2\pi}{B} \Big) + 2 \pi \ZZ
		\quad\text{ for all } j \in \{1, 2, \ldots, r\}.
	\]
	Hence,
	\[
		2 b' \sprod{\theta_0}{w . \rho} \in \Big(-\frac{2 \pi h}{B}, \frac{2\pi h}{B} \Big) + 2\pi \ZZ
		\quad
		\text{ for all } w \in W.
	\]
	We again set $b = 2 b' \lceil N/(2b') \rceil$. Then
	\[
		N \leqslant b \leqslant \min\big\{2 b' N, 2 B^r \big\},
	\]
	and
	\begin{align*}
		2 M b \sprod{\theta_0}{w . \rho} \in \bigg(-\frac{1}{4}, \frac{1}{4}\bigg) + 2 \pi \ZZ
		\quad
		\text{ for all } w \in W,
	\end{align*}
	which implies \eqref{eq:3:22}.

	Combining all the estimates on $S_j$, we deduce that
	\[
		\sum_{k=1}^{4M}m_{k\rho}(bz_0)\geqslant1+\frac12S_1\geqslant1+ 4 M\abs{W}.
	\]
	Hence, $K_{M,b}(z_0) \geqslant 4 \abs{W} M \geqslant M$. This completes the proof of \eqref{eq:3:12}, and the
	theorem follows.
\end{proof}

\section{Applications}
\label{s - applications}
In what follows, we describe some potentially new situations where arithmetic quantum unique ergodicity (AQUE) could be proved, 
using the kernel constructed in Section \ref{s - kernel construction} and taking advantage of the fact that we are no more 
limited by the type of the underlying (reduced) root system. This would require additional progress in various steps of
the known strategies for AQUE, and more particularly in the strategy elaborated by Brooks and Lindenstrauss (for more details,
see the introduction). Nevertheless, an already available result is a positive entropy theorem as proved by Shem-Tov in  \cite{Shem-Tov2022}.

\begin{theorem}
	\label{th - positive entropy}
	Let $G$ and $\Gamma$ be as in Theorem \ref{th - positive entropy intro} of the introduction.
	Assume that $G$ is $\RR$-split. 
	Let $\mu$ be an $A$-invariant probability measure on the homogeneous space $X = \Gamma \backslash G$. Assume that
	$\mu=c \lim_{j\to\infty}|\varphi_j|^2 \: {\rm dvol_{\Gamma \backslash G}}$ is a weak-$\star$ limit where $c>0$ and 
	$\varphi_j$ are some $L^2$-normalized joint eigenfunctions of both $\mathscr{Z}(\mathfrak{g})$ and of the Hecke algebra 
	at some fixed prime $p$. Then any regular element $a\in A$ acts with positive entropy on each ergodic component of $\mu$.
\end{theorem}

This statement is merely a generalization of \cite[Theorem 1.1]{Shem-Tov2022}, passing from ${\rm SL}_n(\RR)$ to any split
simple $\RR$-group. Its proof consists in checking that Shem-Tov's arguments in \cite[Section 4]{Shem-Tov2022} can be used,
once we replace his kernels constructed in \cite[Section 3.1]{Shem-Tov2022} by our kernels from Theorem \ref{thm:3:1}. 
Before going into this specific result, we first describe the general standard context which will allow us to address some 
additional potential applications. 

\subsection{The general setting}
\label{ss - general setting}
In what follows, $\mathbf{H}$ denotes a linear algebraic group defined over a number field $F$ which is assumed to be absolutely 
almost simple and simply connected. We denote by $\Phi$ the (absolute) root system of $\mathbf{H}_{\overline F}$ where
$\overline F$ is an algebraic closure of $F$. The root system $\Phi$ is reduced since $\mathbf{H}_{\overline F}$ is split; 
its Dynkin diagram is connected since $\mathbf{H}_{\overline F}$ is almost simple. 

The starting point is relevant to real homogeneous spaces. Let $H$ be a non-compact simple linear real Lie group. 
Its associated Riemannian symmetric space is $\widetilde M = H/K$ where $K$ is any maximal compact subgroup of $H$. 
The desired quantum unique ergodicity results are expected to hold on the homogeneous space $\Gamma \backslash H$ where
$\Gamma$ is a torsion-free lattice, and have geometric interpretations on the locally symmetric space
$M = \Gamma \backslash H /K$.

Our goal is to explain how to find an $F$-group ${\bf H}$ as above in order to see $H$ as closely related to the $F$-rational
points of ${\bf H}$ for some Archimedean completion of $F$, and also how to construct $S$-arithmetic groups projecting onto
lattices $\Gamma$ in $H$. We will also discuss the issue of cocompactness of $\Gamma$ and explain why, in view of the
classification of non-Archimedean Lie groups, this condition either implies restrictions on the absolute root system $\Phi$ 
or implies to work with several Archimedean places. A convenient reference for all the material below is the first chapter
in Margulis' book \cite{Margulis91}. 

First, if $k$ is a locally compact field containing $F$, it is well-known (see \cite[Chapter I, Proposition 2.3.6]{Margulis91})
that we have the following result, relating an algebraic property for an algebraic group and a topological one for its rational
points in this case. 

\begin{namedthm*}{Theorem}[Anisotropy is compactness]
	The topological group of rational points ${\bf H}(k)$ is compact if and only if the algebraic group ${\bf H}$ has $k$-rank
	$0$.
\end{namedthm*}

Recall that, for any field $E$ containing the ground field $F$, the $E$-rank of ${\bf H}$, denoted by ${\rm rk}_E({\bf H})$, is the dimension of a maximal $E$-split torus in ${\bf H}$; these tori are all conjugate by ${\bf H}(E)$. 

We denote by $\mathscr{T}({\bf H})$ the set of places $v$ of $F$ such that the group ${\bf H}(F_v)$ is compact where $F_v$
is the completion of $F$ with respect to $v$. In other words, $\mathscr{T}({\bf H})$ is the set of places $v$ for which ${\bf H}$
is $F_v$-anisotropic: It is known to be a finite subset of the set $\mathscr{R}$ of all places of $F$. Finally, we denote by
$\mathscr{R}_\infty$ (resp. $\mathscr{R}_{\rm fin}$) the set of Archimedean (resp. non-Archimedean) places of $F$.

For a finite subset $S$ in $\mathscr{R}$, we denote by $F(S)$ the set of $S$-integral elements in $F$, \emph{i.e.} 
the set of elements $x \in F$ such that $|x|_v \leqslant 1$ for each absolute value $| \cdot |_v$ attached to 
$v\in\mathscr{R}_{\rm fin}\smallsetminus S$. 

From now on, we assume that the group ${\bf H}$ is given together with an embedding in some general linear group, which we denote 
by ${\bf H} < {\rm GL}_n$. This allows us to define the subgroup ${\bf H} \bigl( F(S) \bigr)$ to be 
\[
	{\bf H} \bigl( F(S) \bigr) = {\bf H}(F) \cap {\rm GL}_n \bigl( F(S) \bigr). 
\]
This is well-defined up to commensurability in the sense that another embedding ${\bf H} < {\rm GL}_m$ would lead to a 
commensurable subgroup ${\bf H}(F) \cap {\rm GL}_m \bigl( F(S) \bigr)$, \emph{i.e.} a subgroup containing a subgroup of finite 
index in both ${\bf H}(F) \cap {\rm GL}_n \bigl( F(S) \bigr)$ and ${\bf H}(F) \cap {\rm GL}_m \bigl( F(S) \bigr)$.
Any subgroup of the form ${\bf H} \bigl( F(S) \bigr)$ (or, more generally, commensurable with such a subgroup) is called an 
{\it $S$-arithmetic subgroup}~of ${\bf H}(F)$. 

The key result to construct lattices in products of Lie groups is the following theorem, due to Borel--Harish-Chandra and 
Mostow--Tamagawa (see \cite[Chapter I, Theorem 3.2.5]{Margulis91}). 

\begin{namedthm*}{Theorem}[Arithmetic groups are lattices]
	For any finite set $S$ of places of $F$ such that $\mathscr{R}_\infty \smallsetminus\mathscr{T}({\bf H}) \subset S 
	\subset \mathscr{R}$, the $S$-arithmetic group ${\bf H} \bigl( F(S) \bigr)$ is a lattice in the topological group $H_S$ 
	defined to be product group $H_S=\prod_{v\in S}{\bf H}(F_v)$.
\end{namedthm*}
In the next section, we shall see how, from the above statements, one can construct homogeneous spaces $\Gamma\backslash G$ 
when $G$ is an Archimedean Lie group. In addition, this result is complemented by Godement's compactness criterion, 
see \cite[Chapter I, Theorem 3.2.4 (b)]{Margulis91}

\begin{namedthm*}{Theorem}[Cocompactness amounts to anisotropy]
	The $S$-arithmetic lattice ${\bf H} \bigl( F(S) \bigr)$ in the topological group $H_S$ is cocompact if and only if 
	the algebraic group ${\bf H}$ is $F$-anisotropic.
\end{namedthm*}

It is clear that for ${\bf H}$ to be $F$-anisotropic, it is sufficient that ${\bf H}$ be $F_v$-anisotropic, hence equivalently 
that ${\bf H}(F_v)$ be compact, for some place $v$ of $F$. With this remark in mind, since we are mainly targeting a suitable
(cocompact) lattice inclusion $\Gamma < H$, a very standard trick consists then in finding an $F$-group ${\bf H}$ as above 
such that ${\bf H}(\RR)$ is $H$ for some Archimedean completion $\RR$ of $F$, and such that ${\bf H}(F_v)$ is compact for 
some other completion $F_v$, with $v$ Archimedean or not. One limitation of this trick is the fact that there are very few
compact non-Archimedean simple Lie groups, namely (see \emph{e.g.} \cite[Classification tables]{TitsCorvallis79}): 

\begin{namedthm*}{Theorem}[Very few compact non-Archimedean groups]
	The only anisotropic simple groups over non-Archimedean local fields are the groups of inner type $A$, in other words 
	division algebras over non-Archimedean local fields.
\end{namedthm*}

This classification result explains in particular why quaternion algebras, which are fields over $\QQ$, are used to construct
cocompact arithmetic lattices in ${\rm PSL}_2(\RR)$. On the contrary, according to Weyl's unitarian trick
\cite[\S~3 Th\'eor\`eme 1 p. 19]{BourbakiLie9}, any simple real Lie group $H$ has a (unique) anisotropic $\RR$-form which we 
denote by $\mathbf{H}^{\rm cpt}$. Recall that an {\it $F$-form}~of ${\bf H}$ is a linear algebraic $F$-group $\mathscr{H}$
such that ${\bf H} \otimes_F \overline F$ and $\mathscr{H} \otimes_F \overline F$ are isomorphic $\overline F$-groups. 

The last important result we would like to mention in this abstract general context addresses the freedom of choice that is 
available when one wants to find an $F$-group as above with prescribed behavior at finitely many places of $F$. 
The following statement is implied by a deep result in Galois cohomology due to Borel and Harder 
(see \cite[Theorem B]{BorelHarder78}). 

\begin{namedthm*}{Theorem}[Freedom in prescribing local forms]
	Let $\Phi$ be an irreducible reduced root system. Let $S$ be a finite set of places of $F$ and, for each $v \in S$, and 
	let ${\bf H}_v$ be a simply connected $F_v$-group of absolute root system $\Phi$. Then there exists a simply connected 
	$F$-group ${\bf H}$ such that ${\bf H} \otimes_F F_v$ and  ${\bf H}_v$ are isomorphic $F_v$-groups for each $v \in S$. 
\end{namedthm*}

The same statement holds with "adjoint" replacing "simply connected" everywhere. It is precious in the sense that, provided a 
natural compatibility condition is satisfied (namely, sharing the same absolute root system) we can find a suitable $F$-group 
${\bf H}$ in order to construct an $S$-arithmetic lattice ${\bf H} \bigl( F(S) \bigr)$ in $H_S$ as above; moreover, this can
be done while prescribing that at least one factor for some $v \in S$ in $H_S$ be compact, ensuring that the lattice inclusion
${\bf H} \bigl( F(S) \bigr) < H_S$ is cocompact. To reformulate the limitation mentioned after the previous theorem, one 
restriction is that if we choose a compact factor to be non-Archimedean, then the group $F$-group ${\bf H}$ must be an $F$-form
of ${\rm SL}_n$; otherwise, if we choose the compact factor to be Archimedean, then all absolute root systems are allowed but
then $F$ must be a number field $\neq \QQ$ since several Archimedean places are needed. 

\subsection{Potentially new situations for AQUE}
\label{ss - potential applications} 
Using the theoretical framework of the previous subsection, we recall situations when AQUE, or at least positive entropy 
statements, are known to hold. Then we exhibit situations where the kernels from Theorem \ref{thm:3:1} could be used to prove 
new positive entropy results. We freely use here the notation used in the introduction as well as the one of the previous 
subsection. 

We first consider the case when adelic Hecke conditions are imposed, \emph{i.e.} when the functions $\varphi_j$ are 
eigenfunctions of the full Hecke algebra $\mathscr{H}$ (and form a non-degenerate sequence in the sense of \cite{SV07}).
Then it is known that AQUE holds for congruence cocompact lattices in $G_\infty = {\rm SL}_2(\RR)$, see \cite{Lindenstrauss},
and more generally in $G_\infty = {\rm SL}_d(\RR)$ for $d$ prime \cite{SV19}. In the same adelic context, AQUE is also known 
to hold in the non-cocompact case for $G_\infty = {\rm SL}_2(\RR)$, since in \cite{Soundararajan} it is shown that there is
no escape of mass.  

We now turn to the case when local Hecke conditions are imposed, \emph{i.e.} when $(\varphi_j : j\in\NN)$ is a non-degenerate 
sequence of eigenfunctions of a local Hecke algebra $\mathscr{H}_p$ with respect to some specific prime $p$. Then AQUE holds 
for congruence cocompact lattices in $G_\infty = {\rm SL}_2(\RR)$, and cocompact irreducible lattices in
$G_\infty = {\rm SL}_2(\RR) \times {\rm SL}_2(\RR)$, see \cite{BrooksLindenstrauss2010}, and more generally for congruence
cocompact lattices in $G_\infty = {\rm SL}_d(\RR)$ for $d$ prime, see \cite{SilbermanPhD}. In the same local context, AQUE
is also known to hold in the non-cocompact case (up to proportionality, since then escape of mass is not disproved yet)
for $\Gamma =  {\rm SL}_d(\ZZ)$ in $G_\infty = {\rm SL}_d(\RR)$ for $d$ prime, see \cite{Shem-Tov2022}. The latter reference 
contains a positive entropy result elaborating on ideas from \cite{BrooksLindenstrauss2010}. We mention below that thanks to
our kernels and by using Shem-Tov's ideas on ${\rm SL}_k({\ZZ}) \setminus {\rm SL}_k(\RR)$ (where $k$ is any integer 
$\geqslant 2$), the positive entropy result \cite[Theorem 1.1]{Shem-Tov2022} can be generalized from the case 
$G_\infty = {\rm SL}_k(\RR)$ to the case when $G_\infty$ is any split simple real Lie group.

As it was already mentioned in the introduction, deducing an AQUE result from a positive entropy one requires a precise
understanding of measures on the homogeneous spaces $\Gamma\backslash G_\infty$. In the higher rank case, the results in this 
vein that are used in the references above usually come from \cite{EKL} and \cite{EK}. 

Finally, let us explain in which general context our kernels shall be used. Let $G_\infty$ be a non-compact simple linear
real Lie group whose Lie algebra $\mathfrak{g}_\infty$ is absolutely simple (\emph{i.e.} $\mathfrak{g}_\infty \otimes_\RR \CC$ 
is simple). We assume that $G_\infty$ is the group of real points of an absolutely almost simple simply connected $\RR$-group 
${\bf G}$. 
\begin{itemize}[leftmargin=1em]
\item[$\bullet$]
As a first step, by \cite[Proposition 3.8]{Borel63} and after choosing a totally real number field $E$, we can extend 
${\bf G}$ to an $E$-group, still denoted by ${\bf G}$, with ${\bf G}(E) = G_\infty$, and such that the twisted groups
${\bf G}^\sigma(E)$ are compact for all non-trivial embeddings $\sigma : E \to \RR$. Then the group
$\Gamma = {\bf G}(\mathscr{O}_E)$ is a cocompact lattice in $G_\infty \times K_\infty$ where $K_\infty$ is the product of 
the compact Lie groups ${\bf G}^\sigma(E)$. 

\item[$\bullet$]
As a second step, by the Borel-Harder Theorem above and after choosing a non-Archimedean place 
$\frakp \in \mathscr{R}_{\rm fin}$, we can even assume that ${\bf G}$ is split at the completion $E_\frakp$ 
of $E$ with respect to the place $\frakp$. In particular, the Bruhat--Tits building of 
$G_\frakp = {\bf G}(E_\frakp)$ carries a Brooks--Lindenstrauss kernel as constructed in the present paper, since its 
(spherical) root system is reduced. 
\end{itemize}
Setting $S = \mathscr{R}_\infty \cup \{ \frakp \}$ we see that that the $S$-integers consist of the elements in 
$\mathscr{O}_E[\frac1{\frakp}]$ where $\mathscr{O}_E$ is the ring of integers of the number field $E$. Using the notation 
$\Gamma[\frac1{\frakp}] = {\bf G}(\mathscr{O}_E[\frac1{\frakp}])$,
and $K_\frakp={\bf G}(\mathscr{O}_\frakp)$ where $\mathscr{O}_\frakp$ is the valuation ring of the local field
$E_\frakp$, we have $\Gamma = \Gamma[\frac1{\frakp}] \cap K_\frakp$. In addition, the maximal compact subgroup
$K_\frakp$ is actually a maximal proper subgroup since it is a maximal parahoric subgroup of an affine Tits system in 
the simply connected Chevalley group ${\bf G}(E_\frakp)$. This provides the equality ${\bf G}(E_\frakp) 
= \Gamma[\frac1{\frakp}] K_\frakp$ which together with $\Gamma = \Gamma[\frac1{\frakp}] \cap K_\frakp$
allows us to identify (in a similar way as in the Introduction after Theorem \ref{th - positive entropy intro}) the spaces
$\Gamma \backslash (G_\infty \times K_\infty)$ and $\Gamma[\frac1{\frakp}] \backslash
(G_\infty \times K_\infty \times G_\frakp ) / K_\frakp$. Note that not insisting on cocompactness of $\Gamma$ in 
$G_\infty$ allows one to skip the first step and to work with a $\QQ$-group ${\bf G}$, in which case $\frakp$
is merely a prime number $p$. 

We finally use the identification $\Gamma \backslash G_\infty \simeq \Gamma[\frac1{\frakp}] \backslash 
(G_\infty \times G_\frakp ) / K_\frakp$ to see the averaging operators $A_\lambda$ defined in Section 
\ref{ss - analysis on buildings} as Hecke operators acting on $L^2(\Gamma \backslash G_\infty)$. More precisely, at the 
non-Archimedean place $\mathfrak p$ the Cartan decomposition 
$G_{\frakp} = \bigsqcup_{\lambda \in P^+} K_\frakp \lambda K_\frakp$ shows that the characteristic functions 
$\ind{K_\frakp \lambda K_\frakp}$ provide a basis for the space $\calL(G_{\frakp}, K_{\frakp})$ of 
compactly supported bi-$K_\frakp$-invariant functions $f : G_\frakp \to \CC$; the latter space is a unital, 
commutative algebra for convolution (we pick the Haar measure on $G_\frakp$ for which $K_\frakp$ has volume 1). 
It follows from a classical computation that for each $\lambda \in P^+$ the map
$f \mapsto f * \ind{K_\frakp \lambda K_\frakp}$ coincides with the averaging operator $A_\lambda$ on continuous
functions $G_\frakp \to \CC$ which are right $K_\frakp$-invariant. We will use this remark to write 
$A_\lambda \in \calL(G_{\frakp}, K_{\frakp})$. 

To go back to the quotient $\Gamma \backslash G_\infty$, we will call the operator on $L^2(\Gamma \backslash G_\infty)$ 
defined by the partial convolution as 
\[
	f \mapsto \int_{K_\frakp \lambda K_\frakp} f(g_\infty,g_\frakp h^{-1}) \: {\rm d}h,
\]
the \emph{Hecke operator} attached to $\lambda \in P^+$. 
We will denote it again by $A_\lambda$. We also write $A_\lambda \in \calL(G_{\frakp}, K_{\frakp})$ in this context.
The connection between eigenvalues of the Hecke algebra and the spectrum $\Sigma$ is explained in Appendix \ref{app - 1}.

\subsection{Proof of the positive entropy theorem for all $\RR$-split simple groups}
\label{subsection - pf ThB}
We can now turn to an effective application of the kernels constructed in Theorem \ref{thm:3:1}.
We use the notation of Theorem \ref{th - positive entropy} with $G$ equal to $G_\infty$ above; in particular, $\mu=c \lim_{j\to\infty}|\varphi_j|^2 \: {\rm dvol_{\Gamma \backslash G}}$ is a weak-$\star$ limit where $c>0$ and $a$ is a regular element in a maximal $\RR$-split torus $A$. 
We denote by $\pi_\Gamma$ the quotient map $G_\infty \twoheadrightarrow \Gamma\backslash G_\infty = X$. 
At last, we pick an embedding ${\bf G} < {\rm SL}_m$ defined over $\QQ$ so that $G_p = {\bf G} \cap {\rm SL}_m(\QQ_p)$ and
$K_p = {\bf G} \cap {\rm SL}_m(\ZZ_p)$. 
This allows us to define the denominator function ${\bf d}: G_p \to p^\ZZ \cup\{0\}$ by setting ${\bf d}(g)$ to be the maximum of the $p$-adic absolute values of the matrix coefficients of $g$. Notice that ${\bf d}$ is a bi-$K_p$-invariant function on $G_p$. 

In what follows, we basically check that Shem-Tov's proof of \cite[Theorem 1.1]{Shem-Tov2022} works in the context of this section, once one uses our multitemporal wave kernels. 


\begin{proof}
Let $N$ be an integer $\gg 1$. 
We use the positive entropy criterion provided by \cite[Proposition 4.6]{Shem-Tov2022}, which is relevant to pure dynamical systems with an abstract ambient space
$X$ endowed with a probability measure $\mu$ stabilized by a measurable automorphism $a$ (\emph{i.e.} $a_*\mu=\mu$). 
Hence we have to show that the following condition is satisfied: 

\vspace{1ex}
\noindent $(\star)$ 
{\it for any $\eta \in (0;1)$ there exists a partition $\mathcal{P}$ of $X$, an integer $c \in \NN^*$ and $\delta >0$ such that for any subset $J$ of the refined partition $\mathcal{P}^{\vee_acN}$, if $\mu(\bigsqcup_{E \in J}E) \geqslant \eta$ then $|J| \geqslant e^{\delta N}$.}

\vspace{1ex}
This is done in two steps:
\begin{itemize}
    \item[$\bullet$] construction of suitable partitions $\mathcal{P}$ as above;
    \item[$\bullet$] proof of the implication in $(\star)$. 
\end{itemize}
The second step is achieved by analyzing correlations $\langle (\varphi_j \ind{E}) * k , \varphi_j \ind{E} \rangle$, where $k$ is a kernel provided by our Theorem  \ref{thm:3:1} and $E$ belongs to a suitable family of subsets in $X$. 
Roughly speaking, checking the implication in $(\star)$ is done by comparing a lower bound of $\langle (\varphi_j \ind{E}) * k , \varphi_j \ind{E} \rangle$ obtained by spectral properties of $k$ with an upper bound obtained by more computational arguments involving the partition $\mathcal{P}$. 

\vspace{1ex}
\noindent
{\it Construction of suitable partitions}. 
They are given by \cite[Lemma 4.1]{Shem-Tov2022}: 

\begin{lemma} 
\label{lemma-partition}
Let $A$ be a maximal $\RR$-split torus in $G_\infty$ and let $a \in A$ be a regular element. 
Let $\mu$ be an $a$-invariant probability measure on $X$. 
Then for any compact neighbourhood $\Omega_\infty$ of the identity element in $G_\infty$ and for any $r>0$, there exist a sequence
$(B_N : N \in \NN^*)$ of neighbourhoods of the identity element in $G_\infty$, a partition $\mathcal{P}$ of $X$ and an integer $c \in\NN^*$ such that:
\begin{enumerate}[label=\rm (\roman*), start=1, ref=\roman*]
    \item
    \label{en:4:1}
    the intersection with $\pi_\Gamma(\Omega_\infty)$ of any element of the refined partition $\mathcal{P}^{\vee_acN}$ is $\mu$-
    essentially contained in a translate $\pi_\Gamma(xB_N)$ for some $x \in \Omega_\infty$; 
    \item
    \label{en:4:2}
    for any $x,y \in \Omega_\infty$ the number of cosets $bK_p$ in $G_p$ such that $\pi_\Gamma(xB_Nb) \cap \pi_\Gamma(yB_N) \neq \varnothing$ and ${\bf d}(b) \leqslant p^{rN}$ is bounded from above by $N^{O_r(1)}$.
\end{enumerate}
\end{lemma}

\begin{proof}[Proof of Lemma \ref{lemma-partition}]~The ideas are all due to Shem-Tov (see \cite[Section 4]{Shem-Tov2022}) who considered the case of special linear groups but used results valid for more general simple $\RR$-groups; nevertheless, we will point out when $\RR$-splitness of ${\bf G}$ and regularity of $a$ are used. 
We denote by $B(\eta)$ the ball of radius $\eta$ centered at the identity element for an invariant Riemannian metric on $G_\infty$; accordingly, for a subset $C$ in $G_\infty$ we denote by $B(C,\eta)$ its $\eta$-thickening in $G_\infty$.

We pick $\delta>0$. 
By a construction due to Einsiedler-Lindenstrauss \cite[7.51]{EL}, established for simple real Lie groups, there exists a finite entropy partition $\mathcal{P}$ of $X$ containing $X \setminus \pi_\Gamma(\Omega_\infty)$ and such that for any element $E$ of $\mathcal{P}^{\vee_acN}$ with $E \subseteq \pi_\Gamma(\Omega_\infty)$ there exists $x \in \Omega_\infty$ such that $E \subseteq \pi_\Gamma\bigl( x \bigcap_{k=-N}^N a^k B(\delta) a^{-k} \bigr)$, see also \cite[Lemma 4.3]{Shem-Tov2022}. 
Then we use \cite[Lemma 4.4]{Shem-Tov2022}: the proof of this lemma is valid for all simple real Lie groups since it only uses the metric properties of the exponential map, together with the fact that $G$ is $\RR$-split and $a$ is regular, so that
$\mathfrak{g} =  \mathfrak{a} \oplus \mathfrak{g}^+ \oplus \mathfrak{g}$, where $\mathfrak{g}$ and $\mathfrak{a}$ are the Lie algebras of $G_\infty$ and $A$ respectively, and $\mathfrak{g}^\pm$ is the Lie algebra of the contraction group $G^\pm = \{g \in G_\infty : a^kga^{-k} \to 1$ as $k \to \pm\infty\}$ (concretely, the last two groups are the unipotent radicals of the opposite positive and negative Borel subgroups containing $A$ and defined by the choice of $a$). 
Applying the latter lemma allows us to choose $\delta>0$ small enough to find $\alpha>0$ and $\kappa\ll\delta$ such that $\bigcap_{k=-N}^N a^k B(\delta) a^{-k} \subseteq B(C,\kappa e^{-\alpha N})$, where $C$ is a compact subset of $A$.

At this stage, we found a partition $\mathcal{P}$ and $\alpha >0$ such that for any $c \in \NN^*$ and any element $E$ of $\mathcal{P}^{\vee_acN}$ contained in $\Omega_\infty$ we have $E \subset \pi_\Gamma \bigl( xB(C,\kappa e^{-\alpha N}) \bigr)$. 
We set $\varepsilon_N = \kappa e^{-\alpha N}$ and $B_N = B(C,\varepsilon_N)$. 
It remains to show that we may take a possibly smaller $\delta$ and we may adjust $c$ (hence $\varepsilon_N$ and $B_N)$ in order to obtain \eqref{en:4:2}. 
In other words, it remains to count the cosets $bK_p$ such that $\pi_\Gamma(xB_Nb) \cap \pi_\Gamma(yB_N) \neq \varnothing$ and
${\bf d}(bK_p) \leqslant p^{rN}$. 
In view of the identification $\Gamma \backslash G_\infty \simeq \Gamma[\frac1{p}] \backslash (G_\infty \times G_p ) / K_p$, having $\pi_\Gamma(xB_Nb) \cap \pi_\Gamma(yB_N) \neq \varnothing$ amounts to the existence of an element $\gamma \in \Gamma[{1 \over p}]$ such that $\gamma x B_Nb \cap yB_N \neq \varnothing$. 
Looking at the non-Archimedean component of a couple in the latter intersection, we see that $\gamma b \in K_p$ so that ${\bf d}(\gamma) \leqslant p^{O(N)}$. 
Now we enumerate the classes $bK_p$ such that $\pi_\Gamma(xB_Nb) \cap \pi_\Gamma(yB_N) \neq \varnothing$: this gives a list $(b_1K_p, b_2K_p, \dots, b_MK_p)$, and a list $(\gamma_i : i=1\ldots M)$ of corresponding elements in $\Gamma[{1 \over p}]$; finally we set $s_i = \gamma_i \gamma_1^{-1}$ for $i\in\{2;\dots;M\}$:
it is a finite family of elements in $yB_NB_N^{-1}B_NB_N^{-1}y^{-1}\cap \Gamma[{1 \over p}]$. 
It remains to use the counting provided by \cite[Lemma 4.5]{Shem-Tov2022} to obtain the conclusion.
\end{proof}

Let us next estimate correlations $\langle (\varphi_j \ind{E}) * k , \varphi_j \ind{E} \rangle$. 
Recall that $(\varphi_j : j\in\NN)$ is a non-degenerate sequence of normalized $L^2$-eigenfunctions of the local Hecke algebra $\mathscr{H}_p$ and that we consider by assumption a weak-$\star$ limit $\mu=c\lim_{j\to\infty}\mu_i$, where $\mu_i = |\varphi_j|^2 \: {\rm dvol_{\Gamma \backslash G}}$ and $c \geqslant 1$ is a constant such that $\mu$ is probability measure. 
We want to use the positive entropy criterion $(\star)$ above, and for this we pick $\eta \in (0;1)$. 
Let $\Omega_\infty$ be a compact neighbourhood of the identity in $G_\infty$ such that $\mu\bigl(\pi_\Gamma(\Omega_\infty)\bigr)\geqslant 1-{\eta\over2}$.

It follows from Bruhat-Tits theory, more precisely from the geometric interpretation of the Cartan decomposition \cite[3.3]{TitsCorvallis79}, that for each $\gamma \in \Gamma[{1 \over p}]$ the denominator ${\bf d}(\gamma)$ is comparable to the distance from the origin $K_p$ to the vertex $\gamma K_p$ in the affine building of $G_p$, and the latter distance is comparable to the norm of the corresponding vectorial distance $\sigma(K_p,\gamma K_p)$. 
Therefore, in view of the expression \eqref{eq:1:1} of the cardinality of vectorial balls,
Theorem \ref{thm:3:1}\eqref{en:2:1} implies that there is a constant $r>0$ such that the kernel $k=k_{M,N}$ is supported on elements
with denominators $\leqslant p^{rMN}$ for $M$ and $N$ sufficiently large as in the statement of the latter theorem.

We use now the sequence $(B_N : N\in\NN^*)$ of neighbourhoods of the identity element in $G_\infty$, the partition $\mathcal{P}$ of $X$ and the integer $c \in\NN^*$ given by Lemma \ref{lemma-partition} for these choices of $r>0$ and of $\Omega_\infty$. 
Given a union $E=\bigsqcup_{j=1}^dE_j$ of $d$ elements of the refined partition $\mathcal{P}^{\vee_acN}$ such that $\mu(E) \geqslant \eta$, we must show that $d \geqslant e^{\delta N}$ for some suitable $\delta >0$. 

Set ${E}' = {E} \cap \pi_\Gamma(\Omega_\infty)$ and $E_j' = E_j \cap \pi_\Gamma(\Omega_\infty)$ for $1 \leqslant j \leqslant d$. 
Since $\mu(E) \geqslant \eta$ and $\mu \bigl( \pi_\Gamma(\Omega_\infty) \bigr) \geqslant 1 - {\eta \over 2}$, we have $\mu(E') \geqslant {\eta \over 2}$. 
By weak-$\star$ convergence, there is an index $i_0$ such that $c \mu_{i_0}(E')  \geqslant {\eta \over 2}$. 
We are interested in the correlation
\begin{equation}\label{Correlation}
\langle (\varphi_{i_0} \ind{E'}) * k , \varphi_{i_0} \ind{E'} \rangle,
\end{equation}
where $k$ is a wave kernel provided by Theorem \ref{thm:3:1}
with $z_0$ corresponding to $\varphi_{i_0}$ and $M > {2c \over \eta} = \varepsilon^{-1}$.
It follows from the construction of $k$ that the expression \eqref{Correlation} is real
and we shall see that it is actually positive.

\vspace{1ex}
\noindent{\it Estimating the correlation \eqref{Correlation} from above}.
Consider first the inner products $\langle(\varphi_{i_0}\ind{E_j'})*k,\varphi_{i_0}\ind{E_{\ell}'}\rangle$ with
$1\leqslant j, \ell \leqslant d$.
By Lemma \ref{lemma-partition}\eqref{en:4:1}, for any $1 \leqslant j \leqslant d$, there exists $x_j \in G_\infty$ such that the inclusion $E_j' \subseteq \pi_\Gamma(x_j B_N)$ holds $\mu$-essentially. 
Observe that the proof of \cite[Lemma 4.7]{Shem-Tov2022}, stated for special linear groups, remains valid for the groups we consider (it consists in decomposing double classes modulo $K_p$ into right cosets modulo $K_p$ and in applying the Cauchy--Schwarz inequality).
This lemma yields
\begin{align*}
    \bigl|\langle(\varphi_{i_0}\ind{E_j'})*k,\varphi_{i_0}\ind{E_\ell'}\rangle\bigr|
    &\leqslant
    \langle (|\varphi_{i_0}| \ind{E_j' \cap\pi_\Gamma(x_jB_N)})*|k|,|\varphi_{i_0}|
    \ind{E_{\ell}'\cap\pi_\Gamma(x_\ell B_N)} \rangle \\
    &\leqslant 
    v(k)\cdot\|k\|_\infty\cdot\|\varphi_{i_0}\|_{L^2(E_j')} \cdot \|\varphi_{i_0}\|_{L^2(E_{\ell}')},
\end{align*}
where $v(k)$ is the number of classes $bK_p$ with $b\in\Gamma[{1\over p}]$ such that $\bigl(E_j'\cap\pi_\Gamma(x_jB_N)\bigr) b\cap\bigl(E_{\ell}'\cap\pi_\Gamma(x_\ell B_N)\bigr)\neq\varnothing$ and $K_p b K_p \subseteq {\rm supp}(k)$.
By Theorem \ref{thm:3:1}\eqref{en:2:2}, there exists $\delta>0$ such that $\|k\|_\infty \lesssim e^{-\delta N}$ and, by Lemma \ref{lemma-partition}\eqref{en:4:2}, we know that $v(k) \lesssim N^{O_r(1)}$.
Hence,
\begin{equation*}
    \bigl|\langle(|\varphi_{i_0}|\ind{E_j'})*k,|\varphi_{i_0}|\ind{E_{\ell}'}\rangle\bigr|
    \lesssim
    N^{O_r(1)}e^{-\delta N}\|\varphi_{i_0}\|_{L^2(E_j')}\cdot\|\varphi_{i_0}\|_{L^2(E_{\ell}')}
\end{equation*}
for all $1\leqslant j,\ell\leqslant d$. By using the Cauchy--Schwarz inequality for finite sums, we get finally
\begin{align*}
\bigl|\langle(\varphi_{i_0}\ind{E'})*k,\varphi_{i_0}\ind{E'}\rangle\bigr| 
&\lesssim N^{O_r(1)}e^{-\delta N}\biggl(\sum_{j=1}^d\|\varphi_{i_0}\|_{L^2(E_j')}\biggr)^2\\
&\leqslant N^{O_r(1)}e^{-\delta N}d\sum_{j=1}^d\|\varphi_{i_0}\|^2_{L^2(E_j')}
\leqslant N^{O_r(1)}e^{-\delta N}d.
\end{align*}

\noindent
{\it Estimating the correlation \eqref{Correlation} from below}. 
In order to estimate $\langle(\varphi_{i_0}\ind{E'})*k,\varphi_{i_0}\ind{E'}\rangle$ from below, we use the orthogonal decomposition
\begin{equation}\label{OrthogonalDecomposition}
L^2(X)=\CC\varphi_{i_0}\oplus\varphi_{i_0}^\perp,
\end{equation}
so that $\varphi_{i_0}\ind{E'}= \mu_0\varphi_{i_0}+R$, with
\[
\mu_0=\langle \varphi_{i_0}\ind{E'},\varphi_{i_0}\rangle=\int_{E'}|\varphi_{i_0}|^2=\| \vphi_{i_0} \ind{E'}\|^2.
\]
Notice that $\mu_0\geqslant{\eta\over2c}=\varepsilon$ by the initital choice of $\Omega_\infty$.
By Pythagorus, we can estimate the second term (which is a sum of other Hecke eigenfunctions) as follows:
\begin{equation*}
\|R\|^2
=\|\varphi_{i_0}\ind{E'}\|^2-\mu_0^2
\leqslant\mu_0(1-\varepsilon).
\end{equation*}
As the convolution on the right by $k$ preserves the orthogonal decomposition \eqref{OrthogonalDecomposition}, we have
\begin{equation*}
\langle(\varphi_{i_0}\ind{E'})*k,\varphi_{i_0}\ind{E'}\rangle 
=\mu_0^2\langle\varphi_{i_0}*k,\varphi_{i_0}\rangle+\langle R*k,R\rangle.
\end{equation*}
As $\langle\varphi_{i_0}*k,\varphi_{i_0}\rangle\geqslant M>\varepsilon^{-1}$
and $\langle R*k,R\rangle\geqslant-\|R\|^2$,
according to Theorem \ref{thm:3:1}\eqref{en:2:3},
we deduce that
\begin{equation*}
\langle (\varphi_{i_0}\ind{E'})*k,\varphi_{i_0}\ind{E'}\rangle 
\geqslant{\mu_0^2\over\varepsilon}-\mu_0(1-\varepsilon)
=\mu_0\Bigl({\mu_0\over\varepsilon}-1+\varepsilon\Bigr) 
\geqslant\varepsilon^2.
\end{equation*}

\vspace{1ex}\noindent
{\it Conclusion}. By comparing the upper and lower bounds of the correlation \eqref{Correlation}, we finally obtain
\[
N^{O_r(1)}e^{-\delta N}d\gtrsim\varepsilon^2,
\quad\text{\emph{i.e.}},\quad
d\gtrsim\frac{\varepsilon^2e^{\delta N}}{N^{O_r(1)}},
\]
which allows us to apply the positive entropy criterion $(\star)$. 
\end{proof}

\appendix
\section{Hecke algebra eigenvalues and the spectrum $\Sigma$}
\label{app - 1}
In this appendix we retain the notation introduced in Section \ref{s - applications}.
Suppose that $f \in L^2(\Gamma \backslash G_\infty)$ is a joint eigenfunction of the Hecke algebra, that is, there is 
$\psi: P^+ \rightarrow \CC$ such that, for all $\lambda \in P^+$,
\[
	f * A_\lambda = \psi(\lambda) f.
\]
We may assume that $\|f\|_{L^2} = 1$. 
Notice that $\psi$ is a homomorphism from $\calL(G_{\frakp}, K_{\frakp})$ into $\CC$ thus,
in view of \cite[Propositions 1.2.2 \& 1.2.3]{macdo0}, there is $z \in \mathfrak{a}_\CC$ such that
\[
	\psi(\lambda) = P_\lambda(z).
\]
We claim that $P_\lambda(z)$ is positive definite as a function of $\lambda \in P^+$. 

Since both groups $G_\infty$ and $G_{\frakp}$ are of type I, see \emph{e.g.} \cite[Theorem 6.E.19]{BekkaDelaharpe2020}, 
we can use direct integral decompositions. Hence,
\[
	L^2\Big(\Gamma\big[\tfrac{1}{\frakp}\big] \backslash G_\infty \times G_{\frakp} \Big)
	=
	\int^\oplus_{\widehat{G_\infty \times G_{\frakp}}} \calH_\pi \: \mu_{\Gamma[\frac{1}{\frakp}]}({\rm d} \pi).
\]
In particular,
\begin{align*}
	P_\lambda(z) &= \sprod{f * A_\lambda}{f} \\
	&= \int_{\widehat{G_{\infty} \times G_\frakp}}
	\Tr \Bigl(  \pi(f * A_\lambda) \pi(f)^\star\Bigr) \: \mu_{\Gamma[\frac{1}{\frakp}]} ({\rm d} \pi).
\end{align*}
In the formulae above, $\widehat{G_\infty\times G_\frakp}=\widehat{G_\infty}\times\widehat{G_\frakp}$ denotes the unitary dual. To compute the trace, consider an irreducible representation $\pi=\pi_\infty\times\pi_\frakp$ of $G_\infty\times G_\frakp$ on $\calH_\pi=\calH_{\pi_\infty}\otimes\calH_{\pi_\frakp}$. Let us show that $\Tr\bigl(\pi(f*A_\lambda)\pi(f)^\star\bigr)$ vanishes unless $\pi_\frakp$ has nonzero $K_\frakp$-invariant vectors, in which case $\Tr\bigl(\pi(f*A_\lambda)\pi(f)^\star\bigr)$ boils down to a positive-definite Macdonald function. First of all, we have
\begin{align*}
\pi(f * A_\lambda) 
&= \int_{G_\infty \times G_{\frakp}} (f * A_\lambda)(g, g_{\frakp}) \pi(g, g_{\frakp}) 
{\: \rm d}(g, g_{\frakp}) \\
&= \int_{G_\infty \times G_{\frakp}} \int_{G_{\frakp}} f(g, g_\frakp h_\frakp) 
A_\lambda(h_\frakp^{-1}) \pi(g, g_\frakp) 
{\: \rm d} h_\frakp {\: \rm d}(g, g_\frakp) \\
&= \int_{G_\frakp} A_{\lambda}(h_\frakp^{-1}) \int_{G_\infty \times G_\frakp} f(g, g_\frakp) 
\pi(g, g_\frakp) {\: \rm d}(g, g_\frakp) \pi(1, h_\frakp^{-1}) 
{\:\rm d}h_\frakp\\
&=\pi(f)\,{(\Id\otimes\Lambda)},
\end{align*}
where
\[
\Lambda
=\int_{G_\frakp}A_\lambda(h)\pi_\frakp(h){\:\rm d}h
=\fint_{K_\frakp \lambda K_\frakp} \pi_\frakp(h){\:\rm d} h.
\]
{O}bserve that, for every $k \in K_\frakp$,
\[
\pi_\frakp(k)\,\Lambda 
=\fint_{K_\frakp\lambda K_\frakp}\pi_\frakp(kh){\:\rm d}h=\Lambda
\]
and {similarly} $\Lambda=\Lambda\,\pi_\frakp(k)$.
Consequently $\Lambda$ preserves the subspace $\smash{\calH_{\pi_\frakp}^{K_\frakp}}$ of $K_{\frakp}$-fixed vectors in $\calH_{\pi_\frakp}$, and vanishes on the orthogonal subspace.
Hence, $\Lambda=0$ if $\calH_{\pi_\frakp}^{K_\frakp}=\{0\}$.
Otherwise $\calH_{\pi_\frakp}^{K_\frakp}$ is one-dimensional (see \cite[Theorem 1.4.4(ii)]{macdo0}), \emph{i.e.}, $\calH_{\pi_\frakp}^{K_\frakp}=\CC\xi_{\pi_\frakp}$ with $\bigl\|\xi_{\pi_\frakp}\bigr\|=1$.
Moreover, there is $z_{\pi_\frakp} \in \mathfrak{a}_\CC$ such that
\[
	\left\langle \pi_\frakp(h) \xi_{\pi_\frakp}, \xi_{{\pi_\frakp}} \right\rangle
	= P_{\tilde\lambda}(z_{\pi_\frakp}),
\]
for all $\tilde\lambda\in P^+$ and $h \in K_\frakp \tilde\lambda K_\frakp$,
and $P_{\tilde{\lambda}}(z_{\pi_\frakp})$ is positive-definite as a function of $\tilde{\lambda}$,
according to \cite[Proposition 1.4.2(i)]{macdo0}.
{We deduce that}
\[
	\Tr \Bigl( \pi(f * A_\lambda) \pi(f)^\star \Bigr)
    = \sprod{\Lambda \xi_{\pi_\frakp}}{\xi_{\pi_\frakp}} 
	= \fint_{K_\frakp \lambda K_\frakp} 
	\sprod{\pi_\frakp(h) \xi_{\pi_\frakp}}{\xi_{\pi_\frakp}} 
	{\: \rm d} h\\
	= P_\lambda(z_{\pi_\frakp}),
\]
{and we conclude that}
\[
	P_\lambda(z)	
	= \int_{\widehat{G_\infty}\times (\widehat{G_{\frakp}})_{K_\frakp}}
	P_\lambda(z_{\pi_\frakp}) 
	\: \mu_{\Gamma[\frac{1}{p}]} ({\rm d} (\pi_\infty, \pi_\frakp))
\]
is positive-definite by superposition.
Here $(\widehat{G_\frakp})_{K_\frakp}$ denotes the spherical unitary dual, consisting in unitary irreducible representations of $G_\frakp$ with nonzero $K_\frakp$-fixed vectors.

\section*{Declarations}

As requested by the Journal submission system, on behalf of all authors, the corresponding author states that there is no conflict of interest. 
Moreover, data sharing is not applicable to this article as no datasets were generated or analysed during the current study.

\bigskip 

\begin{bibliography}{kernel}
	\bibliographystyle{amsplain}
\end{bibliography}

\end{document}